\documentclass[11pt,a4paper]{article}

\usepackage{studynote}
\newtheorem{thm}{Theorem}[section]
\newtheorem{dfn}{Definition}[section]
\newtheorem{prp}{Proposition}[section]

\newtheorem*{Ass}{Assumption}
\newtheorem{lemma}{Lemma}[section]

\usepackage{indentfirst}

\title{\bf Leader-Follower Linear-Quadratic Stochastic Graphon Games\thanks{This work is financially supported by the National Key R\&D Program of China (2022YFA1006104), National Natural Science Foundations of China (12471419, 12271304), and Shandong Provincial Natural Science Foundation (ZR2024ZD35).}}

\author{\normalsize Weijia Chen\thanks{\it School of Mathematics, Shandong University, Jinan 250100, P.R. China, E-mail: 202520301@mail.sdu.edu.cn},\quad Jingtao Shi\thanks{\it Corresponding author. School of Mathematics, Shandong University, Jinan 250100, P.R. China, E-mail: shijingtao@sdu.edu.cn}}

\date{\today}

\begin{document}

	\maketitle

	\begin{abstract}
	This paper investigates leader-follower linear-quadratic stochastic graphon games, which consist of a single leader and a continuum of followers. The state equations of the followers interact through graphon coupling terms, with their diffusion coefficients depending on the state, the graphon aggregation term, and the control variables. The diffusion term of the leader's state equation depends on its state and control variables. Within this framework, a hierarchical decision-making structure is established: for any strategy adopted by the leader, the followers compete to attain a Nash equilibrium, while the leader optimizes its own cost functional by anticipating the followers' equilibrium response. This work develops a rigorous mathematical model for the game, proves the existence and uniqueness of solutions to the system's state equations under admissible control sets, and constructs a Stackelberg-Nash equilibrium for the continuum follower game. By employing the continuity method, we establish the existence, uniqueness, and stability of solutions to the associated forward-backward stochastic differential equation with a graphon aggregation term.
\end{abstract}

\vspace{2mm}
	
	
\vspace{2mm}
	

%	\tableofcontents
	
	\section{Introduction}
	
	{\it Mean-field games} (MFGs), first introduced by Lasry and Lions \cite{Lasry-Lions-07} and Huang et al. \cite{Huang-Caines-Malhame-07}, provide an approximation method for finite-player games when the number of players is very large. The mean-field approximation holds rigorously when the population of players is sufficiently large and satisfies the conditions of homogeneity and anonymity (homogeneity means that each player has identical system dynamics and cost functional coefficients, while anonymity means that a player does not know which specific other player an influence originates from). This conclusion is referred to as the propagation of chaos result. For the theory of mean-field games, we refer to Bensoussan et al. \cite{Bensoussan-Frehse-Yam-13}, Carmona and Delarue \cite{Carmona-Delarue-18-1,Carmona-Delarue-18-2} and a recent review paper by Laurière \cite{Lauriere-25}.

MFGs provide an approximation for large-population games under the assumptions of homogeneity and anonymity. When players interact through a graph structure, these assumptions are broken. Guéant \cite{Gueant-15} studied MFGs with finite graph interactions, and Delarue \cite{Delarue-17} investigated MFGs with finite random graph interaction structures. Graphon serves as a tool for studying the limit of dense graphs (for a detailed theory, refer to Lovász \cite{Lovasz-12}). Gao and Caines \cite{Gao-Caines-17,Gao-Caines-18,Gao-Caines-19-1,Gao-Caines-19-2,Gao-Caines-20} introduced the graphon structure into control systems. The concept of graphon games/graphon MFGs was proposed for static models by Parise and Ozdaglar \cite{Parise-Ozdaglar-19,Parise-Ozdaglar-23}, and for dynamic models by Caines and Huang \cite{Caines-Huang-18,Caines-Huang-21}. Carmona et al. \cite{Carmona-Cooney-Graves-Lauriere-22} and Aurell et al. \cite{Aurell-Carmona-Lauriere-22} established the connection between graphon games and finite-network games. Bayraktar et al. \cite{Bayraktar-Chakraborty-Wu-23,Bayraktar-Kim-24,Bayraktar-Wu-Zhang-23} provided the convergence from finite graph interaction systems to graphon systems. Coppini et al. \cite{Coppini-DeCrescenzo-Pham-25} studied corresponding results for nonlinear graphon systems. Graphon games find applications in finance \cite{Tangpi-Zhou-24}, epidemic management \cite{Aurell-Carmona-Dayanikli-Lauriere-22}, rumor propagation \cite{Liu-Dayanikli-25}, among other fields.

Stackelberg games (leader-follower games) were initially proposed by von Stackelberg \cite{Stackelberg-34}, describing a hierarchical decision-making problem in economic markets. The players are divided into a leader and followers. The leader first specifies a set of possible controls, and the followers make a best response to each control of the leader. The leader then optimizes its own objective considering the best responses of the followers. Yong \cite{Yong-02} studied stochastic {\it linear-quadratic} (LQ) Stackelberg games. Moon and Ba\c{s}ar \cite{Moon-Basar-18} and Wang \cite{Wang-24} investigated Stackelberg games with one leader and a large number of followers, where under any control of the leader, the followers compete to form a Nash equilibrium. This problem is termed mean-field Stackelberg games (see also Nourian et al. \cite{Nourian-Caines-Malhami-Huang-12}, Wang and Zhang \cite{Wang-Zhang-14}, Bensoussan et al. \cite{Bensoussan-Chau-Yam-15}, Yang and Huang \cite{Yang-Huang-21}, Mukaidani et al. \cite{Mukaidani-Irie-Xu-Zhuang-22}, Dayan{\i}kl{\i} and Lauri\`{e}re \cite{Dayanikli-Lauriere-23}, Si and Shi \cite{Si-Shi-25}) and the references therein.

We first present a brief overview of the research problem addressed in this paper, with a detailed elaboration provided in the subsequent section. We consider an LQ Stackelberg stochastic graphon game involving one leader and a continuum of followers, satisfying the following state equations
\begin{equation}\label{state Intro}
	\begin{cases}
		\mathrm{d}X_t^{f,u}=\big(A_t^{f}X_t^{f,u}+B_t^{f}\alpha_t^{f,u}+C_t^{f}GX_t^{f,u}+b_t^{f}\big)\mathrm{d}t\\
		\qquad\qquad +\big(D_t^{f}X_t^{f,u}+E_t^{f}\alpha_t^{f,u}+F_t^{f}GX_t^{f,u}+\sigma_t^{f}\big)\mathrm{d}W_t^{f,u},\quad u\in I,\\
		\mathrm{d}X_t^l=\big(A_t^lX_t^l+B_t^l\alpha_t^l+C_t^lM_t^f+b_t^l\big)\mathrm{d}t\\
		\qquad\qquad +\big(D_t^lX_t^l+E_t^{l}\alpha_t^l+F_t^lM_t^f+\sigma_t^l\big)\mathrm{d}W_t^l,\\
		X_0^{f,u}=x_0^{f,u},\quad u\in I,\quad X_0^l=x_0^l,
	\end{cases}
\end{equation}
where
\begin{equation}\label{graphon intro}
	GX_t^{f,u}:=\int_I G(u,v)X_t^{f,v}\lambda(\mathrm{d}v),\quad M_t^f:=\int_I X_t^{f,u}\lambda(\mathrm{d}u).
\end{equation}
Here $GX^{f}$ is called the {\it graphon aggregation} of the state variable $X^{f}$. The admissible control sets of the followers and the leader are given by
\begin{equation}\label{admissible controls intro}
	\mathcal{U}^{f}:=L^2_\boxtimes(0,T;\mathbb{R}^{m_1}),\quad \mathcal{U}^l:=L_{\mathcal{F}^l}^2(0,T;\mathbb{R}^{m_2}), 
\end{equation}
respectively. The cost functional of a follower $u\in I$ is defined as
\begin{equation}\label{cost of followers intro}
	J^{f,u}(\alpha^{f,u};\alpha^{f,-u},\alpha^l)=\frac{1}{2}\mathbb{E}\biggl[ \int_0^T\Big\{U^{f,u\top}_t Q_t^{f}U_t^{f,u}+\alpha_t^{f,u\top}R_t^{f}\alpha_t^{f,u}\Big\}\mathrm{d}t+U_T^{f,u\top}G^{f}U_T^{f,u}\biggr],
\end{equation}
where
\begin{equation}\label{U f intro}
	U_t^{f,u}:=\begin{pmatrix}
		X_t^{f,u} \\ GX_t^{f,u} \\ X_t^l
	\end{pmatrix},\quad Q_t^{f}\equiv\begin{pmatrix}
		Q_t^{f,ij}
	\end{pmatrix}_{3\times 3},\quad G^{f}\equiv\begin{pmatrix}
		G^{f,ij}
	\end{pmatrix}_{3\times 3}.
\end{equation}
The cost functional of the leader is
\begin{equation}\label{cost of leader intro}
	J^l(\alpha^l;\alpha^f)=\frac{1}{2}\mathbb{E}\biggl[ \int_0^T\Big\{U_t^{l\top} Q_t^lU_t^l+\alpha_t^{l\top}R_t^l\alpha_t^l\Big\}\mathrm{d}t+U_T^{l\top}G^lU_T^l \biggr],
\end{equation}
where
\begin{equation}\label{U l intro}
	U_t^l:=\begin{pmatrix}
		X_t^l \\ M_t^f
	\end{pmatrix},\quad Q_t^l\equiv\begin{pmatrix}
		Q_t^{l,ij}
	\end{pmatrix}_{2\times 2},\quad G^l\equiv\begin{pmatrix}
		G^{l,ij}
	\end{pmatrix}_{2\times 2},
\end{equation}
The objective of both the leader and the followers is to optimize their respective cost functionals. The decision-making process has a hierarchical structure: the leader first specifies an admissible control $\alpha^l\in\mathcal{U}^l$, then the followers engage in a stochastic graphon game, competing to form a Nash equilibrium $\hat{\alpha}^f(\alpha^l)$. The leader, anticipating the Nash equilibrium that the followers will form, selects $\hat{\hat{\alpha}}^l$ to minimize the cost functional $J^l(\alpha^l,\hat{\alpha}^f(\alpha^l))$. The rigorous mathematical framework and the existence and uniqueness of solutions to the state equations will be provided in Section 3. Section 4 will present the Stackelberg-Nash equilibrium for the problem.

In Section 5, we employ the continuation method to study {\it forward-backward stochastic differential equations} (FBSDEs) with graphon aggregation terms (abbreviated as {\it graphon-aggregated FBSDEs}). In the classical theory of FBSDEs, the continuation method was primarily introduced by Hu and Peng \cite{Hu-Peng-95} and further developed by Yong \cite{Yong-97} and Peng and Wu \cite{Peng-Wu-99}. The core idea is to construct a family of FBSDEs connected to a solvable FBSDE. We introduce corresponding monotonicity conditions and establish the existence and uniqueness of solutions for the graphon-aggregated FBSDEs, along with estimates of the solutions and stability results concerning the interaction graphon. Section 6 provides some concluding remarks.

The contribution of this paper can be summarized as follows.

(1) First, we investigate stochastic graphon games with a leader-follower structure, which is an area that, to the best of our knowledge, has not been systematically explored. A rigorous mathematical framework is established for this problem, which is notable for its generality-specifically, the diffusion terms in the state equations for both the leader and the followers are formulated in a general form.

(2) Second, by employing the maximum principle, we derive a representation of the Stackelberg-Nash equilibrium for the problem.​ Furthermore, we establish sufficient conditions that guarantee the existence and uniqueness of a solution to the graphon-aggregated FBSDEs, which characterizes the equilibrium state of the followers and their associated adjoint processes. Sufficient conditions for the solvability of the problem are also provided.

(3) Third, the continuation method is applied to analyze a general linear graphon-aggregated FBSDEs.​ For this equation, we obtain sufficient conditions ensuring the existence and uniqueness of a solution, derive its $L^2$ estimates, and prove its continuous dependence on the coupling graphon.
	
	\section{Preliminaries}
	
	\subsection{Graphon theory}

For the theory of graphons, we refer to the monograph \cite{Lovasz-12} by Lovász.

Let $I=[0,1]$ and consider $(I,\mathcal{I},\lambda)$ as a probability space. A graphon is a bounded, symmetric, measurable function on $I\times I$ taking values in $[0,1]$. The set of all such graphons is denoted by $\mathcal{W}_0$.

A graphon $W\in\mathcal{W}_0$ is called a \textit{step function} if there exists a finite partition $S_1\cup\cdots\cup S_n$ of $I$ such that $W$ is constant on each block $S_i\times S_j$. The sets $S_i$ are referred to as the \textit{steps} of $W$.

\begin{dfn}\label{def norm distance of graphons}
	For any $G,G_1,G_2\in\mathcal{W}_0$, the following norm and distance are defined:
	\begin{equation}\label{norm distance}
		\|G\|_\infty=\sup_{u,v\in I}G(u,v),\quad d(G_1,G_2)=\|G_1-G_2\|_\infty.
	\end{equation}
\end{dfn}

From the definition, it follows that $0\le \|G\|_\infty\le 1$.

Graphons can be interpreted as weighted operators.

\begin{dfn}\label{def weighted operator of graphons}
	Denote by $L^2_{\mathcal{I}}(\mathbb{R}^n)$ the space of square-integrable $\mathbb{R}^n$-valued functions on $I$. For $G\in\mathcal{W}_0$, define the corresponding bounded linear operator, also denoted by $G$, from $L^2_{\mathcal{I}}(\mathbb{R}^n)$ to $L^2_{\mathcal{I}}(\mathbb{R}^n)$ as
	\begin{equation}\label{weighted operator}
		GX^u=\int_I G(u,v)X^v\lambda(\mathrm{d}v),\quad X\in L^2_{\mathcal{I}}(\mathbb{R}^n).
	\end{equation}
\end{dfn}

The bounded linear operator generated by a graphon (referred to as the graphon operator) possesses the following properties.

\begin{prp}\label{properties of graphon operators}
	Let $G\in\mathcal{W}_0$ be regarded as the graphon operator defined above. Then
	
	(1) For any $u\in I$ and $X\in L^2_{\mathcal{I}}(\mathbb{R}^n)$, we have
	\begin{equation}\label{properties of graphon operators-1}
		GX^u=\int_I G(u,v)X^v\lambda(\mathrm{d}v)\le \|G\|_\infty\cdot\int_I |X^v|\lambda(\mathrm{d}v),
	\end{equation}
	and
	\begin{equation}\label{properties of graphon operators-2}
		\int_I |GX^u|^2\lambda(\mathrm{d}u)\le \|G\|_\infty^2\cdot\int_I |X^u|^2\lambda(\mathrm{d}u)\le \int_I |X^u|^2\lambda(\mathrm{d}u).
	\end{equation}
	
	(2) $G$ is a compact self-adjoint operator. There exists a sequence of eigenvectors $\{\phi_i\}_{i=1}^{\infty}$ of $G$ that forms an orthonormal basis for $\overline{\mathrm{ran}\;G}$. Denoting the corresponding eigenvalues by $\{\lambda_i\}_{i=1}^\infty$, the operator admits the spectral decomposition
	\begin{equation}\label{spectral decomposition-1}
		GX=\sum_{k=1}^\infty \lambda_k\langle X,\phi_k \rangle_I\phi_k,\quad \forall X\in L^2_{\mathcal{I}}(\mathbb{R}^n),
	\end{equation}
	where
	\begin{equation}\label{spectral decomposition-2}
		\langle X,\phi_k \rangle_I=\int_I X^{u\top}\phi_k^u\lambda(\mathrm{d}u).
	\end{equation}
\end{prp}

For further details on the spectral decomposition theorem for compact self-adjoint operators, we refer to Conway \cite{Conway-90}. 

	\subsection{Rich Fubini extension}

\begin{dfn}\label{def independent and uncorrelated}
	Let $(\varOmega,\mathcal{F},P)$ and $(I,\mathcal{I},\lambda)$ be two probability spaces, and let $E$ be a Polish space.
	
	(1) A mapping $X:\varOmega\times I\to E$ is called essentially pairwise independent if, for $\lambda$-a.e. $u\in I$ and $\lambda$-a.e. $v\in I$, the random variables $X^u$ and $X^v$ are independent.
	
	(2) A mapping $X:\varOmega\times I\to E$ is called essentially pairwise uncorrelated if, for $\lambda$-a.e. $u\in I$, $X^u$ is square-integrable, and for $\lambda$-a.e. $u\in I$ and $\lambda$-a.e. $v\in I$, the random variables $X^u$ and $X^v$ are uncorrelated.
\end{dfn}

\begin{dfn}\label{def Fubini extension}
	Let $(\varOmega,\mathcal{F},P)$ and $(I,\mathcal{I},\lambda)$ be two probability spaces. Let $(\varOmega\times I,\mathcal{F}\otimes\mathcal{I},P\otimes\lambda)$ denote the usual product space. A probability space $(\varOmega\times I,\mathcal{W},Q)$ is said to be a Fubini extension of $(\varOmega\times I,\mathcal{F}\otimes\mathcal{I},P\otimes\lambda)$ if, for every integrable real-valued function $f(\omega,u)$ on $(\varOmega\times I,\mathcal{W},Q)$, the following conditions hold:
	
	(1) For $\lambda$-a.e. $u\in I$, the function $f^u(\omega)=f(\omega,u)$ is integrable on $(\varOmega,\mathcal{F},P)$.
	
	(2) For $P$-a.e. $\omega\in\varOmega$, the function $f_\omega(u)=f(\omega,u)$ is integrable on $(I,\mathcal{I},\lambda)$.
	
	(3) The function $\int_\varOmega f^u(\omega)\mathrm{d}P$ is integrable on $(I,\mathcal{I},\lambda)$, and the function $\int_I f_\omega(u)\lambda(\mathrm{d}u)$ is integrable on $(\varOmega,\mathcal{F},P)$. Moreover, the following equality holds:
	\begin{equation}\label{Fubini theorem}
		\int_{\varOmega\times I}f\mathrm{d}Q=\int_I\biggl( \int_\varOmega f^u\mathrm{d}P \biggr)\mathrm{d}\lambda=\int_{\varOmega}\biggl( \int_I f_\omega\mathrm{d}\lambda \biggr)\mathrm{d}P.
	\end{equation}
\end{dfn}

\begin{prp}\label{rich Fubini extension}
	Let $I=[0,1]$, let $\mathcal{B}_I$ be the Borel $\sigma$-algebra on $[0,1]$, and let $\lambda_I$ be the Lebesgue measure on $[0,1]$. Let $E$ be a Polish space. Then there exists an extension $(I,\mathcal{I},\lambda)$ of $(I,\mathcal{B}_I,\lambda_I)$, a probability space $(\varOmega,\mathcal{F},P)$, and a Fubini extension $(\varOmega\times I,\mathcal{F}\boxtimes\mathcal{I},P\boxtimes\lambda)$ of $(\varOmega\times I,\mathcal{F}\otimes\mathcal{I},P\otimes\lambda)$ such that, for any measurable mapping $\varphi:(I,\mathcal{I},\lambda)\to\mathcal{P}(E)$, there exists an $\mathcal{F}\boxtimes\mathcal{I}$-measurable process $f:\varOmega\times I\to E$ for which the family of random variables $\{f^u\}_{u\in I}$ is essentially pairwise independent, and for each $u\in I$, $\mathcal{L}(f^u)=\varphi(u)$.
\end{prp}

\begin{proof}
	This statement appears in \cite{Aurell-Carmona-Dayanikli-Lauriere-22}. A Fubini extension satisfying the conditions of this proposition is also called a \textit{rich Fubini extension}. For related references, see \cite{Sun-06,Sun-Zhang-09,Podczeck-10}.
\end{proof}

The following result is called the \textit{exact law of large numbers}.

\begin{prp}\label{exact law of large numbers}
	Let $X$ be a square-integrable function on the Fubini extension $(\varOmega\times I,\mathcal{F}\boxtimes\mathcal{I},P\boxtimes\lambda)$. Then the following conditions are equivalent:
	
	(1) $X$ is essentially pairwise uncorrelated.
	
	(2) For every $A\in\mathcal{I}$ with $\lambda(A)>0$, we have
	\begin{equation}
		\int_A X^u(\omega)\lambda(\mathrm{d}u)=\int_A \mathbb{E}[X^u]\lambda(\mathrm{d}u),\quad P\text{-}a.s.\;\omega\in\varOmega.
	\end{equation}
\end{prp}

\begin{proof}
	A detailed proof of the theorem can be found on pp. 205 of \cite{Carmona-Delarue-18-1}.
\end{proof}

The following proposition establishes the existence of a family of essentially pairwise independent Brownian motions and their initial conditions.

\begin{prp}\label{existence of essentially pairwise independent BM}
	Let $(\varOmega\times I,\mathcal{F}\boxtimes\mathcal{I},P\boxtimes\lambda)$ be the rich Fubini extension as in Proposition \ref{rich Fubini extension}. Let $\upsilon:I\to\mathcal{P}_2(\mathbb{R})$ be an $\mathcal{I}$-measurable function. Then there exists a mapping $\mathbb{B}=(B,x_0):\varOmega\times I\to\mathcal{C}\times\mathbb{R}$ such that $\mathbb{B}$ is essentially pairwise independent, for each $u\in I$, $B^u$ is a one-dimensional standard Brownian motion, and $\mathcal{L}(x_0^u)=\upsilon(u)$.
\end{prp}

\begin{proof}
	Let $T>0$ and denote $\mathcal{C}=C(0,T;\mathbb{R})$. Then $\mathcal{C}\times\mathbb{R}$ is a Polish space with the $\sigma$-algebra $\mathcal{B}(\mathcal{C})\otimes\mathcal{B}(\mathbb{R})$. For each $u\in I$, define $\varphi(u)=\mu(u)\otimes\upsilon(u)$, where $\mu(u)$ denotes the Wiener measure on $\mathcal{C}$. By Proposition \ref{rich Fubini extension}, there exists an $\mathcal{F}\boxtimes\mathcal{I}$-measurable process $\mathbb{B}:\varOmega\times I\to\mathcal{C}\times\mathbb{R}$, $\mathbb{B}=(B,x_0)$, that is essentially pairwise independent, and for each $u\in I$, $\mathcal{L}(B^u,x_0^u)=\mu(u)\otimes\upsilon(u)$.
\end{proof}

For convenience, we shall refer to $B\equiv(B^u)_{u\in I}$ as a family of essentially pairwise independent Brownian motions. 
	
	\section{Problem formulation}
	
	\subsection{Notations}

Let $A^\top$ denote the transpose of a matrix $A$. The set of all $n \times n$ real symmetric matrices is denoted by $\mathcal{S}^n$, the set of all positive semidefinite matrices by $\mathcal{S}^n_+$, and the set of all positive definite matrices by $\hat{\mathcal{S}}^n_+$. For a matrix $Q \in \mathcal{S}^n$ and a constant $K \ge 0$, we write $Q \ge K$ if $x^\top Q x \ge K |x|^2$ for all $x \in \mathbb{R}^n$. For an $\mathcal{S}^n$-valued process $Q = (Q_t)_{t \in [0,T]}$ and a constant $K \ge 0$, we write $Q \gg K$ if for every $t \in [0,T]$ and all $x \in \mathbb{R}^n$, $x^\top Q_t x \ge K |x|^2$.

Let $(I, \mathcal{I}, \lambda)$, $(\varOmega, \mathcal{F}, P)$, and $(\varOmega \times I, \mathcal{F} \boxtimes \mathcal{I}, P \boxtimes \lambda)$ be the extension, sample space, and the corresponding rich Fubini extension, respectively, that satisfy the conclusions of Proposition \ref{rich Fubini extension}. Denote by $\mathbb{E}$ the mathematical expectation on $(\varOmega, \mathcal{F}, P)$, and by $\mathbb{E}^\boxtimes$ the mathematical expectation on $(\varOmega \times I, \mathcal{F} \boxtimes \mathcal{I}, P \boxtimes \lambda)$. Let $(B, x_0)$ be the essentially pairwise independent Brownian motions and initial states constructed in Proposition \ref{existence of essentially pairwise independent BM}. For each $u \in I$, the filtration $\{\mathcal{F}_t^u\}_{t \in [0,T]}$ is defined by
\begin{equation}\label{filtration}
	\mathcal{F}_t^u = \sigma(x_0^u, W_s^u, 0 \le s \le t) \vee \mathcal{N},
\end{equation}
where $\mathcal{N}$ denotes the collection of all $P$-null sets.

We define the following function spaces.

$L^\infty(0,T;\mathbb{R}^n)$: the set of all bounded functions from $[0,T]$ to $\mathbb{R}^n$.

$C^\infty(0,T;\mathbb{R}^n)$: the set of all continuous functions from $[0,T]$ to $\mathbb{R}^n$.

$L_\mathcal{I}^0(0,T;\mathbb{R}^n)$: the set of all $\mathbb{R}^n$-valued, $\mathcal{I}$-measurable processes on $[0,T]$.

$L_\mathcal{I}^2(0,T;\mathbb{R}^n)$: the set of all $X \in L_\mathcal{I}^0(0,T;\mathbb{R}^n)$ satisfying $\int_I \int_0^T |X_t^u|^2 \mathrm{d}t \lambda(\mathrm{d}u) < \infty$.

$L_{\mathcal{I}}^\infty(0,T;\mathbb{R}^n)$: the set of all bounded processes in $L_{\mathcal{I}}^0(0,T;\mathbb{R}^n)$.

$L_\boxtimes^0(0,T;\mathbb{R}^n)$: the set of all $\mathbb{R}^n$-valued, $\mathcal{F} \boxtimes \mathcal{I}$-measurable processes $X = (X_t^u)_{u \in I, t \in [0,T]}$ such that for each $u \in I$, $X^u$ is adapted to the filtration $\{\mathcal{F}_t^u\}_{t \in [0,T]}$.

$L_\boxtimes^2(0,T;\mathbb{R}^n)$: the set of all $X \in L_\boxtimes^0(0,T;\mathbb{R}^n)$ satisfying $\mathbb{E}^\boxtimes\big[ \int_0^T |X_t|^2 \mathrm{d}t \big] < \infty$.

$L_\boxtimes^2(\varOmega \times I; C(0,T;\mathbb{R}^n))$: the set of all $X \in L_\boxtimes^2(0,T;\mathbb{R}^n)$ that are continuous in $t$ and satisfy $\mathbb{E}^\boxtimes\big[ \sup_{0 \le t \le T} |X_t|^2 \big] < \infty$.

$\bar{L}_\boxtimes^2(0,T;\mathbb{R}^n)$: the set of all almost surely determined processes in $L_\boxtimes^2(0,T;\mathbb{R}^n)$, i.e., for every $X \in \bar{L}_\boxtimes^2(0,T;\mathbb{R}^n)$, there exists $\tilde{X} \in L_\mathcal{I}^2(0,T;\mathbb{R}^n)$ such that $X_t^u(\omega) = \tilde{X}_t^u, \forall t \in [0,T], P \boxtimes \lambda\text{-a.e.}$.

$L_\boxtimes^\infty(0,T;\mathbb{R}^n)$: the set of all bounded processes in $L_\boxtimes^0(0,T;\mathbb{R}^n)$.

$L_{\boxtimes}^2(\mathbb{R}^n)$: the set of all $\mathcal{F} \boxtimes \mathcal{I}$-measurable, $\mathbb{R}^n$-valued random variables.

For a sub-$\sigma$-algebra $\mathcal{G}$ of $\mathcal{F}$, $L_{\mathcal{G}}^2(\mathbb{R}^n)$ denotes the set of all $\mathcal{G}$-measurable, $\mathbb{R}^n$-valued square-integrable random variables.

For a filtration $\{\mathcal{H}_t\}_{t \in [0,T]}$ on $\mathcal{F}$, $L_{\mathcal{H}}^2(0,T;\mathbb{R}^n)$ denotes the set of all $\{\mathcal{H}_t\}$-adapted, $\mathbb{R}^n$-valued processes on $[0,T]$ satisfying $\mathbb{E}\big[ \int_0^T |X_t|^2 \, \mathrm{d}t \big] < \infty$.

$L_{\boxtimes_T}^2(\mathbb{R}^n)$: the set of all $X \in L_\boxtimes^2(\mathbb{R}^n)$ such that for each $u \in I$, $X^u \in L_{\mathcal{F}_T}^2(\mathbb{R}^n)$.

In the estimates for the equations, we use the abbreviations $\mathbb{H}_\boxtimes^2(\mathbb{R}^n) = L_\boxtimes^2(0,T;\mathbb{R}^n)$ and $\mathbb{S}_\boxtimes^2(\mathbb{R}^n) = L_\boxtimes^2(\varOmega \times I; C(0,T;\mathbb{R}^n))$ for brevity. 

	\subsection{Formulation of the problem}

Let $(\varOmega\times I,\mathcal{F}\boxtimes\mathcal{I},P\boxtimes\lambda)$ be the rich Fubini extension constructed in Proposition \ref{rich Fubini extension}, where $(\varOmega,\mathcal{F},P)$ is the sample space, and $(I,\mathcal{I},\lambda)$ is the index space for the followers. Let $W^f\equiv(W^{f,u})_{u\in I}$ be a family of essentially pairwise independent Brownian motions, and $W^l$ be a Brownian motion independent of all $W^{f,u}$.

The superscript $f$ denotes the system and control variables for the followers, while the superscript $l$ denotes those for the leader. Assume that each follower's state is $n_1$-dimensional, the leader's state is $n_2$-dimensional, each follower's control is $m_1$-dimensional, the leader's control is $m_2$-dimensional, and all Brownian motions are one-dimensional. The system dynamics for the followers and the leader are given by:
\begin{equation}\label{state}
	\begin{cases}
		\mathrm{d}X_t^{f,u}=\big(A_t^{f}X_t^{f,u}+B_t^{f}\alpha_t^{f,u}+C_t^{f}GX_t^{f,u}+b_t^{f}\big)\mathrm{d}t\\
		\qquad\qquad +\big(D_t^{f}X_t^{f,u}+E_t^{f}\alpha_t^{f,u}+F_t^{f}GX_t^{f,u}+\sigma_t^{f}\big)\mathrm{d}W_t^{f,u},\quad u\in I,\\
		\mathrm{d}X_t^l=\big(A_t^lX_t^l+B_t^l\alpha_t^l+C_t^lM_t^f+b_t^l\big)\mathrm{d}t\\
		\qquad\qquad +\big(D_t^lX_t^l+E_t^{l}\alpha_t^l+F_t^lM_t^f+\sigma_t^l\big)\mathrm{d}W_t^l,\\
		X_0^{f,u}=x_0^{f,u},\quad u\in I,\quad X_0^l=x_0^l,
	\end{cases}
\end{equation}
where $GX_t^{f,u}$ and $M_t^f$ are defined in (\ref{graphon intro}).

The admissible control sets for the followers and the leader are defined as:
\begin{equation*}
	\mathcal{U}^{f}=L^2_\boxtimes(0,T;\mathbb{R}^{m_1}),\quad \mathcal{U}^l=L_{\mathcal{F}^l}^2(0,T;\mathbb{R}^{m_2}),
\end{equation*}
respectively. 
The cost functional of a follower is:
\begin{equation}\label{cost of followers}
	J^{f,u}(\alpha^{f,u};\alpha^{f,-u},\alpha^l)=\frac{1}{2}\mathbb{E}\biggl[ \int_0^T\Big\{U^{f,u\top}_t Q_t^{f}U_t^{f,u}+\alpha_t^{f,u\top}R_t^{f}\alpha_t^{f,u}\Big\}\mathrm{d}t+U_T^{f,u\top}G^{f}U_T^{f,u}\biggr],
\end{equation}
where $U_t^{f,u}, Q_t^{f}, G^{f}$ are defined in (\ref{U f intro}). The cost functional of the leader is:
\begin{equation}\label{cost of leader}
	J^l(\alpha^l;\alpha^f)=\frac{1}{2}\mathbb{E}\biggl[ \int_0^T\Big\{U_t^{l\top} Q_t^lU_t^l+\alpha_t^{l\top}R_t^l\alpha_t^l\Big\}\mathrm{d}t+U_T^{l\top}G^lU_T^l \biggr],
\end{equation}
where $U_t^l, Q_t^l, G^l$ are defined in (\ref{U f intro}).

Each follower and the leader seek to optimize their respective cost functionals. We consider a hierarchical decision structure: the leader first specifies the admissible control set $\mathcal{U}^l$. For each control $\alpha^l\in\mathcal{U}^l$ chosen by the leader, the followers compete to attain a Nash equilibrium $\hat{\alpha}^f(\alpha^l)\in\mathcal{U}^f$. The leader, anticipating the followers' Nash equilibrium, then selects $\hat{\hat{\alpha}}^l\in\mathcal{U}^l$ to optimize the cost functional $J^l(\alpha^l,\hat{\alpha}^f(\alpha^l))$.

The Stackelberg-Nash equilibrium for the leader and the followers will be solved in Section 4. The remainder of this section is devoted to establishing the well-posedness of the problem framework. 

	\subsection{Existence and uniqueness of solutions of state equation}

This section establishes that for any admissible control, the follower's system admits a unique solution, and the corresponding graphon aggregation is deterministic.

The following result extends the domain of the graphon operator to the space of square-integrable functions $L_\boxtimes^2(\mathbb{R}^n)$ on the rich Fubini extension.

\begin{prp}\label{extension of domain of graphon operator}
	Let $G\in\mathcal{W}_0$. Define the operator generated by the graphon as
	\begin{equation*}
		GX^u:=\int_I G(u,v)X^v\lambda(\mathrm{d}v),\quad X\in L_\boxtimes^2(\mathbb{R}^n).
	\end{equation*}
	Then $G$ is a bounded linear operator from $L_\boxtimes^2(\mathbb{R}^n)$ to $L_\boxtimes^2(\mathbb{R}^n)$, and its operator norm is at most $||G||_\infty$.
\end{prp}

\begin{proof}
	For the proof that $G$ is a bounded linear operator, we refer to \cite{Aurell-Carmona-Lauriere-22}. By Proposition \ref{properties of graphon operators}, we have
	\begin{equation*}
		\mathbb{E}^\boxtimes[|GX|^2]\le ||G||_\infty^2\cdot \mathbb{E}^\boxtimes[|X|^2],
	\end{equation*}
	which implies that the operator norm of $G$ is bounded by $||G||_\infty$.
\end{proof}

\begin{Ass}[A1]\label{assumptions of coefficients}
	The following coefficients are assumed to satisfy:
	\begin{equation*}
		\begin{cases}
			A^f,C^f,D^f,F^f\in L^\infty(0,T;\mathbb{R}^{n_1\times n_1}),\quad B^f,E^f\in L^\infty(0,T;\mathbb{R}^{n_1\times m_1}),\\
			A^l,D^l\in L^\infty(0,T;\mathbb{R}^{n_2\times n_2}),\quad B^l,E^l\in L^\infty(0,T;\mathbb{R}^{n_2\times m_2}),\\
			C^l,F^l\in L^\infty(0,T;\mathbb{R}^{n_2\times n_1}),\quad b^{f},\sigma^f\in L^2(0,T;\mathbb{R}^{n_1}),\quad b^l,\sigma^l\in L^2(0,T;\mathbb{R}^{n_2}),\\
			Q^f\in L^\infty(0,T;\mathbb{S}_+^{2n_1+n_2}),\quad R^f\in L^\infty(0,T;\mathbb{S}_+^{m_1}),\quad G^f\in \mathbb{S}_+^{2n_1+n_2},\\
			Q^l\in L^\infty(0,T;\mathbb{S}_+^{n_1+n_2}),\quad R^l\in L^\infty(0,T;\mathbb{S}_+^{m_2}),\quad G^l\in \mathbb{S}_+^{n_1+n_2}.
		\end{cases}
	\end{equation*}
\end{Ass}

\begin{prp}\label{wellposedness of SDE}
	Assume that assumption (A1) holds. For any $\alpha^f\in\mathcal{U}^f$ and $z\in L_{\boxtimes}^2(0,T;\mathbb{R}^n)$, the equation
	\begin{equation}\label{SDE followers}
		\begin{cases}
			\mathrm{d}X_t^{f,u}=\big(A_t^{f}X_t^{f,u}+B_t^{f}\alpha_t^{f,u}+C_t^{f}z^{u}_t+b_t^{f}\big)\mathrm{d}t\\
			\qquad\qquad +\big(D_t^{f}X_t^{f,u}+E_t^{f}\alpha_t^{f,u}+F_t^{f}z^{u}_t+\sigma_t^{f}\big)\mathrm{d}W_t^{f,u},\quad u\in I,\\
			X_0^{f,u}=x_0^{f,u},\quad u\in I,
		\end{cases}
	\end{equation}
	admits a unique solution $X^f\in\mathbb{S}_\boxtimes^2(\mathbb{R}^n)$.
\end{prp}

\begin{proof}
	Interpreting $W^f$ as a standard Brownian motion on the rich Fubini extension $(\varOmega\times I,\mathcal{F}\boxtimes\mathcal{I},P\boxtimes\lambda)$, the conclusion follows from classical SDE theory (for example, Chapter 1 of \cite{Yong-Zhou-99}). 
\end{proof}

For a fixed $\alpha^f\in\mathcal{U}^f$, define the mapping $\varPhi:L_{\boxtimes}^2(0,T;\mathbb{R}^n)\to L_{\boxtimes}^2(0,T;\mathbb{R}^n)$ by
\begin{equation*}
	\varPhi(z):=GX,\quad z\in L_{\boxtimes}^2(0,T;\mathbb{R}^n),
\end{equation*}
where $X$ is the solution to equation (\ref{SDE followers}) with parameters $(\alpha^f,z)$.

\begin{prp}\label{fixed point}
	For any fixed $\alpha^f\in\mathcal{U}^f$, the mapping $\varPhi$ has a unique fixed point. In particular, the restriction of $\varPhi$ to the subspace $\bar{L}_\boxtimes^2(0,T;\mathbb{R}^n)$ also possesses a unique fixed point.
\end{prp}

\begin{proof}
	Let $z^1,z^2\in L_\boxtimes^2(0,T;\mathbb{R}^n)$, and denote by $X^1,X^2$ the solutions to equation (\ref{SDE followers}) with parameters $(\alpha^f,z^1)$ and $(\alpha^f,z^2)$, respectively. Define $\hat{z}:=z^1-z^2$ and $\hat{X}:=X^1-X^2$. Then $\hat{X}$ satisfies
	\begin{equation*}
		\begin{cases}
			\mathrm{d}\hat{X}_t^u=\big(A_t^f\hat{X}_t+C_t^f\hat{z}_t)\mathrm{d}t+(D_t^f\hat{X}_t+F_t^f\hat{z}_t\big)\mathrm{d}W_t^{f,u},\quad u\in I,\\
			\hat{X}^u_0=0,\quad u\in I.
		\end{cases}
	\end{equation*}
	By standard SDE estimates, there exists a constant $K>0$ such that for any $t\in[0,T]$,
	\begin{equation*}
		\mathbb{E}^\boxtimes\Big[ \sup_{0 \le s\le t}|\hat{X}_s|^2 \Big]\le K\mathbb{E}^\boxtimes\Big[ \int_0^t |\hat{z}_s|^2\mathrm{d}s \Big].
	\end{equation*}
	By Proposition \ref{extension of domain of graphon operator}, we have
	\begin{multline*}
		\sup_{0\le s\le t}\mathbb{E}^\boxtimes\big[|\varPhi(z^1)_s-\varPhi(z^2)_s|^2\big]\le \mathbb{E}^\boxtimes\Big[ \sup_{0\le s\le t}|\varPhi(z^1)_s-\varPhi(z^2)_s|^2 \Big]\\
		=\mathbb{E}^\boxtimes\Big[ \sup_{0\le s\le t}|G(X^1_s-X^2_s)|^2 \Big]\le \mathbb{E}^\boxtimes\Big[ \sup_{0\le s\le t}|X^1_s-X^2_s|^2 \Big]\le K\int_0^t \mathbb{E}^\boxtimes\big[|z_s^1-z_s^2|^2\big]\mathrm{d}s.
	\end{multline*}
	Proceeding by mathematical induction, let $\varPhi^{(k)}$ denote the $k$-th iterate of $\varPhi$. Then for any $t>0$ and $k\ge 1$,
	\begin{equation*}
		\sup_{0\le s\le t}\mathbb{E}^\boxtimes\big[|\varPhi^{(k)}(z^1)_s-\varPhi^{(k)}(z^2)_s|^2\big]\le K^{k+1}\frac{t^k}{k!}\int_0^t \mathbb{E}^\boxtimes\big[|z_s^1-z_s^2|^2\big]\mathrm{d}s.
	\end{equation*}
	For sufficiently large $k$, $\varPhi^{(k)}$ becomes a contraction mapping. Hence, $\varPhi$ itself has a unique fixed point.
	
	When $z\in \bar{L}_{\boxtimes}^2(0,T;\mathbb{R}^n)$, the solution $X$ to equation (\ref{SDE followers}) is essentially pairwise independent. By Proposition \ref{exact law of large numbers},
	\begin{equation*}
		\varPhi(z)^u_t=\int_I G(u,v)X_t^v\lambda(\mathrm{d}v)=\int_I G(u,v)\mathbb{E}[X_t^v]\lambda(\mathrm{d}v),
	\end{equation*}
	which implies $\varPhi(z)\in\bar{L}_{\boxtimes}^2(0,T;\mathbb{R}^n)$. As in the previous argument, the restriction of $\varPhi$ to $\bar{L}_\boxtimes^2(0,T;\mathbb{R}^n)$ is a contraction and thus has a unique fixed point.
\end{proof}

\begin{thm}\label{existence and uniqueness of solutions of state equation}
	Assume that assumption (A1) holds. For any $\alpha^f\in\mathcal{U}^f$ and $\alpha^l\in\mathcal{U}^l$, (\ref{state}) admits a unique solution $(X^f,X^l)\in \mathbb{S}_\boxtimes^2(\mathbb{R}^n)\times L_{\mathcal{F}^l}^2(\varOmega;C(0,T;\mathbb{R}^n))$. Moreover, the graphon aggregation satisfies $GX^f\in\bar{L}_{\boxtimes}^2(0,T;\mathbb{R}^n)$, and $M^f$ is deterministic.
\end{thm}

\begin{proof}
	The existence and uniqueness of the solution follow from Proposition \ref{wellposedness of SDE}, Proposition \ref{fixed point}, and classical SDE theory. By Proposition \ref{fixed point}, $X^f$ is essentially pairwise independent, and $GX^f\in\bar{L}_\boxtimes^2(0,T;\mathbb{R}^n)$. Applying Proposition \ref{exact law of large numbers} yields
	\begin{equation*}
		M_t^f=\int_I X_t^u\lambda(\mathrm{d}u)=\int_I \mathbb{E}[X_t^u]\lambda(\mathrm{d}u).\qedhere
	\end{equation*}
\end{proof} 

	\section{Solving Stackelberg games with a continuum of followers}
	
	\subsection{The followers' problem}

\begin{dfn}\label{Nash equilibrium of the followers' game}
	For a given $\alpha^l\in\mathcal{U}^l$, we say that $\hat{\alpha}^f\in\mathcal{U}^f$ is a Nash equilibrium of the followers' problem if for every $u\in I$ and every $\beta^u\in L_{\mathcal{F}^u}^2(0,T;\mathbb{R}^{m_1})$, we have
	\begin{equation*}
		J^{f,u}(\hat{\alpha}^{f,u};\hat{\alpha}^{f,-u},\alpha^l)\le J^{f,u}(\beta^u;\hat{\alpha}^{f,-u},\alpha^l).
	\end{equation*}
\end{dfn}

By Theorem \ref{existence and uniqueness of solutions of state equation}, for any $\alpha^l\in\mathcal{U}^l$ and any $u\in I$, $X^{f,u}$ is independent of $X^l$, and $GX^{f,u}$ is deterministic. Consequently, the cost functional (\ref{cost of followers}) can be rewritten as
\begin{multline}\label{cost of followers rewritten as}
	J^{f,u}(\alpha^{f,u};\alpha^{f,-u},\alpha^l)=\frac{1}{2}\mathbb{E}\biggl[ \int_0^T \Big\{V_t^{f,u\top}Q_t^{f}V_t^{f,u}+\mathrm{tr}\big\{Q_t^{f,33}\varSigma_t\big\}+\alpha_t^{f,u\top}R_t^{f}\alpha_t^{f,u}\Big\}\mathrm{d}t\\
	+V_T^{f,u\top}G^{f}V_T^{f,u}+\mathrm{tr}\big\{G^{f,33}\varSigma_T\big\} \biggr],
\end{multline}
where $\varSigma_t$ denotes the covariance matrix of $X_t^l$, and
\begin{equation}\label{V f u}
	V_t^{f,u}:=\begin{pmatrix}
		X_t^{f,u} \\ GX_t^{f,u} \\ \mathbb{E}[X_t^l]
	\end{pmatrix}.
\end{equation}
Note that $\varSigma_t$ is deterministic. Thus, in (\ref{cost of followers rewritten as}), the influence of $X^l$ is replaced by its expectation $\mathbb{E}[X^l]$.

In the followers' problem, the interaction among followers is defined through integrals. Since altering the value at a single point does not affect the integral, each follower can treat the graphon aggregation $GX^{f,u}$ and the leader's state $X^l$ as external variables when deviating individually. Therefore, for each follower $u$, the problem reduces to a stochastic optimal control one, which can be approached via the maximum principle.

For each follower $u$, define the Hamiltonian function as
\begin{multline}\label{Hamiltiniaon function followers}
	H^{f,u}(t,x,\alpha,p,q):=p^\top\big(A_t^{f}x+B_t^{f}\alpha+C_t^{f}GX_t^{f,u}+b_t^{f,u}\big)+q^\top \big(D_t^{f}x+E_t^{f}\alpha\\
	+F_t^{f}GX_t^{f,u}+\sigma_t^{f,u}\big)+\frac{1}{2}\big( v_t^{f,u\top}Q_t^{f}v_t^{f,u}+\mathrm{tr}\{Q_t^{f,33}\varSigma_t\}+\alpha^\top R_t^{f}\alpha \big),
\end{multline}
where
\begin{equation}\label{v f u}
	v_t^{f,u}:=\begin{pmatrix}
		x \\ GX_t^{f,u} \\ \mathbb{E}[X_t^l]
	\end{pmatrix}.
\end{equation}

\begin{Ass}[A2]
	$E^f\in C(0,T;\mathbb{R}^{n_1\times m_1})$, $R^f\in C(0,T;\mathbb{S}^{m_1}_+)$, $E^l\in C(0,T;\mathbb{R}^{n_2\times m_2})$, and $R^l\in C(0,T;\mathbb{S}^{m_2}_+)$.
\end{Ass}

Under these conditions, for each $u\in I$, the follower's problem is a standard stochastic LQ optimal control problem. By stochastic control theory (see Chapter 6 of \cite{Yong-Zhou-99}), we have the following result.

\begin{thm}\label{maximum principle followers problem}
	Assume that assumptions (A1) and (A2) hold. For a fixed $\alpha^l\in\mathcal{U}^l$ and a given $GX^f\in \bar{L}_\boxtimes^2(0,T;\mathbb{R}^{n_1})$, the followers' problem admits a unique Nash equilibrium $\hat{\alpha}^f$. Denote the corresponding state of the followers by $\hat{X}^f$, then there exist unique adjoint processes $(p^f,q^f)$ such that $(\hat{X}^f,p^f,q^f,\hat{\alpha}^f)$ satisfies the Hamiltonian system
	\begin{equation}\label{Hamiltiniaon system followers}
		\begin{cases}
			\mathrm{d}\hat{X}_t^{f,u}=\big(A_t^{f}\hat{X}_t^{f,u}+B_t^{f}\hat{\alpha}_t^{f,u}+C_t^{f}GX_t^{f,u}+b_t^{f}\big)\mathrm{d}t\\
			\qquad\qquad +\big(D_t^{f}\hat{X}_t^{f,u}+E_t^{f}\hat{\alpha}_t^{f,u}+F_t^{f}GX_t^{f,u}+\sigma_t^{f}\big)\mathrm{d}W_t^{f,u},\quad u\in I,\\
			\mathrm{d}p_t^{f,u}=-\big(A_t^{f\top}p_t^{f,u}+D_t^{f\top}q_t^{f,u}+Q_t^{f,11}\hat{X}_t^{f,u}+Q_t^{f,12}GX_t^{f,u}+Q^{f,13}_t\mathbb{E}[X_t^l]\big)\mathrm{d}t\\
			\qquad\qquad+q_t^{f,u}\mathrm{d}W_t^{f,u},\quad u\in I,\\
			\hat{X}_0^{f,u}=x_0^{f,u},\quad p_T^{f,u}=G^{f,11}\hat{X}_T^{f,u}+G^{f,12}GX_T^{f,u}+G^{f,13}\mathbb{E}[X_t^l],\quad u\in I,
		\end{cases}
	\end{equation}
	and the Nash equilibrium is given by
	\begin{equation}\label{Nash equilibrium followers problem}
		R_t^f\hat{\alpha}_t^{f,u}+B_t^{f\top}p_t^{f,u}+E_t^{f\top}q_t^{f,u}=0,\quad u\in I.
	\end{equation}
\end{thm}

\begin{proof}
	This follows directly from standard results for stochastic LQ optimal control problems (see Theorem 4.2, Proposition 5.5, and Corollary 5.7 in Chapter 6 of \cite{Yong-Zhou-99}).
\end{proof}

If the equilibrium state $\hat{X}^f$ satisfies
\begin{equation}\label{aggregate equilibrium state}
	GX^f=G\hat{X}^f,\quad P\boxtimes\lambda\text{-}a.e.,
\end{equation}
then $\hat{\alpha}^f$ is a Nash equilibrium of the followers' problem under the leader's control $\alpha^l$.

Substituting (\ref{Nash equilibrium followers problem}) and (\ref{aggregate equilibrium state}) into the Hamiltonian system (\ref{Hamiltiniaon system followers}) yields
\begin{equation}\label{graphon-aggregated FBSDE followers}
\left\{
	\begin{aligned}
		\mathrm{d}\hat{X}_t^{f,u}&=\big[A_t^{f}\hat{X}_t^{f,u}-B_t^{f}(R_t^{f})^{-1}B_t^{f\top}p_t^{f,u}-B_t^{f}(R_t^{f})^{-1}E_t^{f\top}q_t^{f,u}\\
                                 &\qquad +C_t^{f}G\hat{X}_t^{f,u}+b_t^{f,u}\big]\mathrm{d}t +\big[D_t^{f}\hat{X}_t^{f,u}-E_t^{f}(R_t^{f})^{-1}B_t^{f\top}p_t^{f,u}\\
                                 &\qquad -E_t^{f}(R_t^{f})^{-1}E_t^{f\top}q_t^{f,u}+F_t^{f}G\hat{X}_t^{f,u}+\sigma_t^{f,u}\big]\mathrm{d}W_t^{f,u},\quad u\in I,\\
		      \mathrm{d}p_t^{f,u}&=-\big[A_t^{f\top}p_t^{f,u}+D_t^{f\top}q_t^{f,u}+Q_t^{f,11}\hat{X}_t^{f,u}+Q_t^{f,12}G\hat{X}_t^{f,u}\\
                                 &\qquad +Q_t^{f,13}\mathbb{E}[X_t^l])\mathrm{d}t+q_t^{f,u}\mathrm{d}W_t^{f,u},\quad u\in I,\\
		\hat{X}_0^{f,u}&=x_0^{f,u},\quad p_T^{f,u}=G^{f,11}\hat{X}_T^{f,u}+G^{f,12}G\hat{X}_T^{f,u}+G^{f,13}\mathbb{E}[X_T^l],\quad u\in I,
	\end{aligned}
\right.
\end{equation}
which is graphon-aggregated FBSDE. We give its solvability and uniqueness of (\ref{graphon-aggregated FBSDE followers}), under the following conditions, by a more general result Theorem \ref{existence and uniqueness of GMFLFBSDE} in Section 5.

\begin{Ass}[A3]
	There exists a constant $K>0$ such that
	\begin{gather*}
		Q^{f,11}\gg K,\quad \begin{pmatrix}
			B^f \\ E^f
		\end{pmatrix}(R^f)^{-1}\begin{pmatrix}
			B^f \\ E^f
		\end{pmatrix}^\top \gg K,\\
		(1+3||G||_\infty)\cdot\sup_{0\le t\le T}\big( ||Q_t^{f,12}||_\infty+||C_t^f||_{\infty}+||F_t^f||_\infty \big)<2K,\\
		\int_I x^{u\top}G^{f,12}Gx^{u}\lambda(\mathrm{d}u)\ge 0,\quad \forall x\in L_\mathcal{I}^2(\mathbb{R}^n).\label{eq33}
	\end{gather*}
\end{Ass}

\begin{thm}\label{solvability of graphon-aggregated FBSDE followers}
	If assumptions (A1), (A2), and (A3) hold, then the graphon-aggregated FBSDE (\ref{graphon-aggregated FBSDE followers}) admits a unique solution $(\hat{X}^f,p^{f},q^f)\in\mathbb{S}_\boxtimes^2(\mathbb{R}^n)\times\mathbb{S}_\boxtimes^2(\mathbb{R}^n)\times\mathbb{H}_\boxtimes^2(\mathbb{R}^n)$, satisfying
	\begin{multline}\label{estimate of graphon-aggregated FBSDE followers}
		\mathbb{E}^\boxtimes\biggl[\sup_{0\le t\le T}|\hat{X}_t^f|^2+\sup_{0\le t\le T}|p^f_t|^2+\int_0^T |q_t^f|^2\mathrm{d}t \biggr]\\
		\le K\mathbb{E}^\boxtimes\biggl[ \int_0^T \big\{ |b_t^f|^2+|\sigma_t^f|^2+|X_t^l|^2 \big\}\mathrm{d}t+|X_T^l|^2 \biggr].
	\end{multline}
\end{thm}

\begin{proof}
	We will discuss the solvability of linear graphon-aggregated FBSDEs later (Theorem \ref{existence and uniqueness of GMFLFBSDE}). Here we directly apply that result, provided we verify the assumptions (S1) and (S2) in that theorem.
	
	For any $t\in[0,T]$ and $x,y,z\in\mathbb{R}^n$, by the first three inequalities in (A3), we have
	\begin{multline}\label{monotonicity condition-followers}
		\begin{pmatrix}
			x \\ y \\ z
		\end{pmatrix}^\top\begin{pmatrix}
			-Q_t^{f,11} & -A_t^{f\top} & -D_t^{f\top} \\
			A_t^f & -B_t^f(R_t^f)^{-1}B_t^{f\top} & -B_t^f(R_t^f)^{-1}E_t^{f\top} \\
			D_t^f & -E_t^f(R_t^f)^{-1}B_t^{f\top} & -E_t^f(R_t^f)^{-1}E_t^{f\top}
		\end{pmatrix}\begin{pmatrix}
			x \\ y \\ z
		\end{pmatrix}\\
		=-x^\top Q_t^{f,11}x-\begin{pmatrix}
			y \\ z
		\end{pmatrix}^\top\begin{pmatrix}
			B^f_t \\ E_t^f
		\end{pmatrix}(R_t^f)^{-1}\begin{pmatrix}
			B^f_t \\ E_t^f
		\end{pmatrix}^\top\begin{pmatrix}
			y \\ z
		\end{pmatrix}\le -K(|x|^2+|y|^2+|z|^2),
	\end{multline}
	and there exists a constant $0<K_2<K$ such that
	\begin{equation*}
		(1+3||G||_\infty)\cdot\sup_{0\le t\le T}\big( ||Q_t^{f,12}||_\infty+||C_t^f||_{\infty}+||F_t^f||_\infty \big)\le 2K_2<2K.
	\end{equation*}
	Moreover, for any $t\in[0,T]$ and $x,y,z\in L_{\mathcal{I}}^2(\mathbb{R}^n)$,
	\begin{equation}\label{inequality}
        \begin{aligned}
		&\int_I \big\{ -x^{u\top} Q_t^{f,12}Gx^u+y^{u\top} C_t^f Gx^u+z^{u\top} F_t^f Gx^u \big\}\lambda(\mathrm{d}u)\\
		&\le \sup_{0\le t\le T}\big( ||Q_t^{f,12}||_\infty+||C_t^f||_\infty+||F_t^f||_\infty \big)\\
        &\qquad \cdot\int_I\big\{ |x^{u\top} Gx^u|+|y^{u\top} Gx^u|+|z^{u\top} Gx^u| \}\lambda(\mathrm{d}u)\\
		&\le \frac{1}{2}\sup_{0\le t\le T}\big( ||Q_t^{f,12}||_\infty+||C_t^f||_\infty+||F_t^f||_\infty \big)\\
        &\qquad \cdot\int_I \big\{ |x^u|^2+|y^u|^2+|z^u|^2+3|Gx^u|^2 \big\}\lambda(\mathrm{d}u)\\
		&\le \frac{1+3||G||_\infty}{2}\sup_{0\le t\le T}\big( ||Q_t^{f,12}||_\infty+||C_t^f||_\infty+||F_t^f||_\infty \big)\\
        &\qquad \cdot\int_I \big\{ |x^u|^2+|y^u|^2+|z^u|^2 \big\}\lambda(\mathrm{d}u)\\
		&\le K_2\int_I\big\{ |x^u|^2+|y^u|^2+|z^u|^2 \big\}\lambda(\mathrm{d}u).
	    \end{aligned}
   \end{equation}
	From (\ref{monotonicity condition-followers}), (\ref{inequality}), and the last inequality in (A3), together with assumptions (A1) and (A2), it follows that conditions (S1) and (S2) hold.
\end{proof}

\begin{thm}\label{thm06}
	Under assumptions (A1), (A2), and (A3), for any $\alpha^l\in\mathcal{U}^l$, the followers' problem admits a unique Nash equilibrium.
\end{thm}

\begin{proof}
	This follows directly from Theorems \ref{maximum principle followers problem} and \ref{solvability of graphon-aggregated FBSDE followers}.
\end{proof}

Assume the solution to the forward-backward system (\ref{graphon-aggregated FBSDE followers}) has the following form:
\begin{equation}\label{relation p and X}
	p_t^{f,u}=P_t^{f,u}\hat{X}_t^{f,u}+\hat{\varphi}_t^{f,u},\quad u\in I,
\end{equation}
where $P^{f}$ and $\hat{\varphi}^f$ are deterministic functions.

Substituting (\ref{relation p and X}) into (\ref{graphon-aggregated FBSDE followers}) and comparing the coefficients of the $\mathrm{d}W_t^{f,u}$ term in the backward equation yields
\begin{multline*}
	\big[I+P_t^{f,u}E_t^f(R_t^f)^{-1}E_t^{f\top}\big]q_t^{f,u}=P_t^{f,u}\big[D_t^f-E_t^f(R_t^f)^{-1}B_t^{f\top}P_t^{f,u}\big]\hat{X}_t^{f,u}\\
	-P_t^{f,u}E_t^f(R_t^f)^{-1}B_t^{f\top}\hat{\varphi}_t^{f,u}+P_t^{f,u}F_t^fG\hat{X}_t^{f,u}+P_t^{f,u}\sigma_t^{f,u}.
\end{multline*}
Assuming that $\hat R_t^f\equiv I+P_t^{f,u}E_t^f(R_t^f)^{-1}E_t^{f\top}$ is invertible, we have
\begin{equation}\label{q}
	\begin{aligned}
		q_t^{f,u}&=(\hat R_t^f)^{-1}P_t^{f,u}\big[D_t^f-E_t^f(R_t^f)^{-1}B_t^{f\top}P_t^{f,u}\big]\hat{X}_t^{f,u}\\
		&\qquad-(\hat R_t^f)^{-1}P_t^{f,u}E_t^f(R_t^f)^{-1}B_t^{f\top}\hat{\varphi}_t^{f,u}+(\hat R_t^f)^{-1}P_t^{f,u}F_t^f\big[G\hat{X}_t^{f,u}+\sigma_t^{f,u}\big].
	\end{aligned}
\end{equation}
Substituting (\ref{relation p and X}) and (\ref{q}) into (\ref{graphon-aggregated FBSDE followers}) and comparing the coefficients of the $\mathrm{d}t$ term in the backward equation, setting the coefficient of $X_t^{f,u}\mathrm{d}t$ to zero, we obtain the Riccati equation
\begin{equation}\label{Riccati followers}
	\begin{cases}
		\dot P_t^{f,u}+A_t^{f\top} P_t^{f,u}+P_t^{f,u}A_t^f-P_t^{f,u}B_t^{f}(R_t^f)^{-1}B_t^{f\top}P_t^{f,u}+\big[D_t^{f\top}-P_t^{f,u}B_t^{f}(R_t^f)^{-1}E_t^{f\top}\big]\\
		\qquad\cdot\big[I+P_t^{f,u}E_t^f(R_t^f)^{-1}E_t^{f\top}\big]^{-1}P_t^{f,u}\big[D_t^f-E_t^f(R_t^f)^{-1}B_t^{f\top}P_t^{f,u}\big]+Q_t^{f,11}=0,\\
		P_T^{f,u}=G^{f,11},
	\end{cases}
\end{equation}
and the forward-backward system
\begin{equation}\label{forward-backward system}
	\begin{cases}
		\mathrm{d}\hat{X}_t^{f,u}=\big(\hat{A}_t^{f,u}\hat{X}_t^{f,u}+\hat{B}_t^{f,u}\hat{\varphi}_t^{f,u}+\hat{C}_t^{f,u}G\hat{X}_t^{f,u}+\hat{b}_t^{f,u}\big)\mathrm{d}t\\
		\qquad\qquad +\big(\hat{D}_t^{f,u}\hat{X}_t^{f,u}+\hat{E}_t^{f,u}\hat{\varphi}_t^{f,u}+\hat{F}_t^{f,u}G\hat{X}_t^{f,u}+\hat{\sigma}_t^{f,u}\big)\mathrm{d}W_t^{f,u},\\
		\mathrm{d}\hat{\varphi}_t^{f,u}=\big(\hat{H}_t^{f,u}\hat{\varphi}_t^{f,u}+\hat{I}_t^{f,u}G\hat{X}_t^{f,u}-Q_t^{f,13}\mathbb{E}[X_t^l]+\hat{g}_t^{f,u}\big)\mathrm{d}t,\\
		\hat{X}_0^{f,u}=x_0^{f,u},\quad \hat{\varphi}_T^{f,u}=G^{f,12}G\hat{X}_T^{f,u}+G^{f,13}\mathbb{E}[X_T^l],
	\end{cases}
\end{equation}
where
\begin{equation*}
	\begin{aligned}
		\hat{A}_t^{f,u}&:=A_t^f-B_t^f(R_t^{f})^{-1}B_t^{f\top}P_t^{f,u}-B_t^f(R_t^f)^{-1}E_t^{f\top}(\hat R_t^f)^{-1}P_t^{f,u}\\
        &\qquad\qquad \cdot\big[D_t^f-E_t^f(R_t^f)^{-1}B_t^{f\top}P_t^{f,u}\big],\\
		\hat{B}_t^{f,u}&:=-B_t^f(R_t^{f})^{-1}B_t^{f\top}+B_t^{f}(R_t^f)^{-1}E_t^{f\top}(\hat R_t^f)^{-1}P_t^{f,u}E_t^f(R_t^f)^{-1}B_t^{f\top},\\
		\hat{C}_t^{f,u}&:=C_t^f-B_t^f(R_t^f)^{-1}E_t^{f\top}(\hat R_t^f)^{-1}P_t^{f,u}F_t^f,\\
		\hat{b}_t^{f,u}&:=b_t^f-B_t^f(R_t^f)^{-1}E_t^{f\top}(\hat R_t^f)^{-1}P_t^{f,u}\sigma_t^{f},\\
		\hat{D}_t^{f,u}&:=D_t^f-E_t^f(R_t^f)^{-1}B_t^{f\top}P_t^{f,u}-E_t^f(R_t^f)^{-1}E_t^{f\top}(\hat R_t^f)^{-1}P_t^{f,u}\\
		&\qquad\qquad \cdot\big[D_t^f-E_t^f(R_t^f)^{-1}B_t^{f\top}P_t^{f,u}\big],\\
		\hat{E}_t^{f,u}&:=-E_t^f(R_t^f)^{-1}B_t^{f\top}+E_t^f(R_t^f)^{-1}E_t^{f\top}(\hat R_t^f)^{-1}P_t^{f,u}E_t^f(R_t^f)^{-1}B_t^{f\top},\\
		\hat{F}_t^{f,u}&:=F_t^f-E_t^f(R_t^{f})^{-1}E_t^{f\top}(\hat R_t^f)^{-1}P_t^{f,u}F_t^f,\\
		\hat{\sigma}_t^{f,u}&:=\sigma_t^{f,u}-E_t^f(R_t^f)^{-1}E_t^{f\top}(\hat R_t^f)^{-1}P_t^{f,u}\sigma_t^{f,u},\\
	\end{aligned}
\end{equation*}
\begin{equation*}
	\begin{aligned}
		\hat{H}_t^{f,u}&:=-A_t^{f\top}+P_t^{f,u}B_t^f(R_t^{f})^{-1}B_t^{f\top}+\big[D_t^{f\top}-P_t^{f,u}B_t^f(R_t^f)^{-1}E_t^{f\top}\big]\\
		&\qquad\qquad \cdot(\hat R_t^f)^{-1}P_t^{f,u}E_t^f(R_t^f)^{-1}B_t^{f\top},\\
		\hat{I}_t^{f,u}&:=-Q_t^{f,12}-P_t^{f,u}C_t^f+\big[-D_t^{f\top}+P_t^{f,u}B_t^f(R_t^f)^{-1}E_t^{f\top}\big](\hat R_t^f)^{-1}P_t^{f,u}F_t^f,\\
		\hat{g}_t^{f,u}&:=-P_t^{f,u}b_t^{f,u}+\big[-D_t^{f\top}+P_t^{f,u}B_t^{f}(R_t^f)^{-1}E_t^{f\top}\big](\hat R_t^f)^{-1}P_t^{f,u}\sigma_t^{f,u}.
	\end{aligned}
\end{equation*}

Equation (\ref{Riccati followers}) is a Riccati equation. Using classical Woodbury matrix identity (see \cite{Golub-van Loan-13}), it is equivalent to the following Riccati equation:
\begin{equation}\label{Riccati followers by Woodbury matrix identity}
	\begin{cases}
		\dot P_t^{f,u}+P_t^{f,u}A_t^{f}+A_t^{f\top}P_t^{f,u}+D_t^{f\top}P_t^{f,u}D_t^{f}+Q^{f,11}_t-\big[B_t^{f\top}P_t^{f,u}+E_t^{f\top}P_t^{f,u}D_t^f\big]^\top\\
		\qquad\qquad \cdot\big[R_t^{f}+E_t^{f\top}P_t^{f,u}E_t^f\big]^{-1}\big[B_t^{f\top}P_t^{f,u}+E_t^{f\top}P_t^{f,u}D_t^f\big]=0,\\
		P_T^{f}=G^{f,11},
	\end{cases}
\end{equation}
By Theorem 7.2, Chapter 6 in \cite{Yong-Zhou-99}, the Riccati equation (\ref{Riccati followers by Woodbury matrix identity}) admits a unique solution $P^f$. Since the coefficients in (\ref{Riccati followers by Woodbury matrix identity}) do not depend on $u$, the solution $P^f$ is independent of $u$. Hence, for any $u\in I$, we denote $P^f=P^{f,u}$ and similarly for other coefficients.

\begin{thm}\label{solution followers problem}
	Under assumptions (A1), (A2), and (A3), and assuming that $\hat R_t^f\equiv I+P_t^{f,u}E_t^f(R_t^f)^{-1}E_t^{f\top}$ is invertible, then for any $\alpha^l\in\mathcal{U}^l$, the followers' problem admits a unique Nash equilibrium
	\begin{equation}\label{Nash equilibrium followers problem-state feedback}
		\begin{aligned}
			\hat{\alpha}_t^{f,u}&=-(R_t^f)^{-1}\big[B_t^{f\top}P_t^{f}+E_t^{f\top}(\hat R_t^f)^{-1}P_t^f(D_t^f-E_t^f(R_t^f)^{-1}B_t^{f\top}P_t^f)\big]\hat{X}_t^{f,u}\\
			&\qquad -(R_t^f)^{-1}\big[B_t^{f\top}-E_t^{f\top}(\hat R_t^f)^{-1}P_t^fE_t^f(R_t^f)^{-1}B_t^{f\top}\big]\hat{\varphi}_t^{f,u}\\
			&\qquad -(R_t^f)^{-1}E_t^{f\top}(\hat R_t^f)^{-1}P_t^f\big[F_t^fG\hat{X}_t^{f,u}+\sigma_t^f\big].
		\end{aligned}
	\end{equation} 
where $P^f$ satisfies the Riccati equation (\ref{Riccati followers}), and $(\hat{X}^{f},\hat{\varphi}^{f})$ satisfies the forward-backward system (\ref{forward-backward system}).
\end{thm} 

\subsection{The leader's problem}

\begin{Ass}[A4]
	There exists a constant $c\in[0,1]$ such that for every $u\in I$,
	\begin{equation}\label{assumption leader}
		\int_I G(u,v)\lambda(\mathrm{d}v)=c.
	\end{equation}
\end{Ass}

\begin{prp}\label{leader's condition}
	Assume condition (A4) holds. Then
	\begin{equation}\label{condition leader}
		\int_I GX_t^{f,u}\lambda(\mathrm{d}u)=cM_t^f.
	\end{equation}
\end{prp}

\begin{proof}
	By interchanging the order of integration,
	\begin{equation*}
		\int_I GX_t^{f,u}\lambda(\mathrm{d}u)=\int_I \biggl( \int_I G(u,v)X_t^{f,v}\lambda(\mathrm{d}v) \biggr)\lambda(\mathrm{d}u)=c\int_I X_t^{f,v}\lambda(\mathrm{d}v)=cM_t^f.\qedhere
	\end{equation*}
\end{proof}

Define
\begin{equation}\label{M t f and N t f}
	\hat{M}_t^f:=\int_I \hat{X}_t^{f,u}\lambda(\mathrm{d}u),\quad \hat{N}_t^f:=\int_I \hat{\varphi}_t^{f,u}\lambda(\mathrm{d}u).
\end{equation}
Integrating the forward-backward system (\ref{forward-backward system}) with respect to the follower's index $u$ yields
\begin{equation}\label{FBODE}
	\begin{cases}
		\mathrm{d}\hat{M}_t^f=\big[\big(\hat{A}_t^f+c\hat{C}_t^f\big)\hat{M}_t^f+\hat{B}_t^f\hat{N}_t^f+\hat{b}_t^f\big]\mathrm{d}t,\\
		\mathrm{d}\hat{N}_t^f=\big(c\hat{I}_t^f\hat{M}_t^f+\hat{H}_t^f\hat{N}_t^f-Q_t^{f,13}\mathbb{E}[X_t^l]+\hat{g}_t^f\big)\mathrm{d}t,\\
		\hat{M}_0^f=\int_I x_0^{f,u}\lambda(\mathrm{d}u),\quad \hat{N}_T^f=cG^{f,12}\hat{M}_t^f+G^{f,13}\mathbb{E}[X_T^l],
	\end{cases}
\end{equation}
which is a two-point boundary value problem for {\it ordinary differential equations} (ODEs). Assume that
\begin{equation}\label{relation N and M}
	\hat{N}_t^f=\hat{P}_t^f\hat{M}_t^f+\hat{N}_t^l.
\end{equation}
Substituting (\ref{relation N and M}) into (\ref{FBODE}) and comparing coefficients, we obtain a Riccati equation
\begin{equation}\label{Riccati leader}
	\begin{cases}
		\dot{\hat{P}}_t^f+\hat{P}_t^f\big(\hat{A}_t^f+c\hat{C}_t^f\big)-\hat{H}_t^f\hat{P}_t^f+\hat{P}_t^f\hat{B}_t^f\hat{P}_t^f-c\hat{I}_t^f=0,\\
		\hat{P}_T^f=cG^{f,12},
	\end{cases}
\end{equation}
 and a forward-backward system:
\begin{equation}\label{FBODE leader}
	\begin{cases}
		\mathrm{d}\hat{M}_t^f=\big[\big(\hat{A}_t^f+c\hat{C}_t^f+\hat{B}_t^f\hat{P}_t^f\big)\hat{M}_t^f+\hat{B}_t^f\hat{N}_t^l+\hat{b}_t^f \big]\mathrm{d}t,\\
		\mathrm{d}\hat{N}_t^l=\big[\big(\hat{H}_t^f-\hat{P}_t^f\hat{B}_t^f\big)\hat{N}_t^l+\big(-Q_t^{f,13}\mathbb{E}[X_t^l]+\hat{g}_t^f-\hat{P}_t^f\hat{b}_t^f\big) \big]\mathrm{d}t,\\
		\hat{M}_0^f=\int_I x_0^{f,u}\lambda(\mathrm{d}u),\quad \hat{N}_T^l=G^{f,13}\mathbb{E}[X_T^l].
	\end{cases}
\end{equation}

Equation (\ref{Riccati leader}) is an asymmetric Riccati equation, independent of $\hat{M}^f$ and $\hat{N}^f$ of (\ref{FBODE}), and can be solved separately (see \cite{Moon-Basar-18}).

\begin{thm}
	If the Riccati equation (\ref{Riccati leader}) admits a unique solution $\hat{P}^f\in C(0,T;\mathbb{R}^{n_1\times n_1})$, then the forward-backward system (\ref{FBODE leader}) has a unique solution $(\hat{M}^f,\hat{N}^l)$, and consequently, the forward-backward system (\ref{FBODE}) also has a unique solution $(\hat{M}^f,\hat{N}^f)$.
\end{thm}

Decomposing the leader's state $X^l$ as follows:
\begin{equation}\label{decompose leader's state}\hspace{-4mm}
	\begin{cases}
		\mathrm{d}\mathbb{E}[X_t^l]=\big\{A_t^l\mathbb{E}[X_t^l]+B_t^l\mathbb{E}[\alpha_t^l]+C_t^l\hat{M}_t^f+\mathbb{E}[b_t^l]\big\}\mathrm{d}t,\\
		\mathrm{d}(X_t^l-\mathbb{E}[X_t^l])=\big\{A_t^l(X_t^l-\mathbb{E}[X_t^l])+B_t^l(\alpha_t^l-\mathbb{E}[\alpha_t^l])+(b_t^l-\mathbb{E}[b_t^l])\big\}\mathrm{d}t\\
		\quad +\big\{D_t^l\mathbb{E}[X_t^l]+D_t^l(X_t^l-\mathbb{E}[X_t^l])+E_t^l\mathbb{E}[\alpha_t^l]+E_t^l(\alpha_t^l-\mathbb{E}[\alpha_t^l])+F_t^lM_t^f+\sigma_t^l\big\}\mathrm{d}W_t^l,\\
		\mathbb{E}[X_0^l]=\mathbb{E}[x_0^l],\quad X_0^l-\mathbb{E}[X_0^l]=x_0^l-\mathbb{E}[x_0^l],
	\end{cases}
\end{equation}
and combining the forward-backward systems (\ref{FBODE leader}) and (\ref{decompose leader's state}), we have
\begin{equation}\label{leader's state equation}\hspace{-4mm}
	\begin{cases}
		\mathrm{d}\hat{M}_t^f=\big[ \big(\hat{A}_t^f+c\hat{C}_t^f+\hat{B}_t^f\hat{P}_t^f\big)\hat{M}_t^f+\hat{B}_t^f\hat{N}_t^l+\hat{b}_t^f \big]\mathrm{d}t,\\
		\mathrm{d}\mathbb{E}[X_t^l]=\big\{A_t^l\mathbb{E}[X_t^l]+B_t^l\mathbb{E}[\alpha_t^l]+C_t^l\hat{M}_t^f+\mathbb{E}[b_t^l]\big\}\mathrm{d}t,\\
		\mathrm{d}(X_t^l-\mathbb{E}[X_t^l])=\big\{A_t^l(X_t^l-\mathbb{E}[X_t^l])+B_t^l(\alpha_t^l-\mathbb{E}[\alpha_t^l])+(b_t^l-\mathbb{E}[b_t^l])\big\}\mathrm{d}t\\
		\quad +\big\{D_t^l\mathbb{E}[X_t^l]+D_t^l(X_t^l-\mathbb{E}[X_t^l])+E_t^l\mathbb{E}[\alpha_t^l]+E_t^l(\alpha_t^l-\mathbb{E}[\alpha_t^l])+F_t^lM_t^f+\sigma_t^l\big\}\mathrm{d}W_t^l,\\
		\mathrm{d}\hat{N}_t^l=\big[ \big(\hat{H}_t^f-\hat{P}_t^f\hat{B}_t^f\big)\hat{N}_t^l+\big(-Q_t^{f,13}\mathbb{E}[X_t^l]+\hat{g}_t^f-\hat{P}_t^f\hat{b}_t^f\big) \big]\mathrm{d}t,\\
		\hat{M}_0^f=\int_I x_0^{f,u}\lambda(\mathrm{d}u),\quad \mathbb{E}[X_0^l]=\mathbb{E}[x_0^l],\\
        X_0^l-\mathbb{E}[X_0^l]=x_0^l-\mathbb{E}[x_0^l],\quad \hat{N}_T^l=G^{f,13}\mathbb{E}[X_T^l].
	\end{cases}
\end{equation}
Define
\begin{gather*}
	\hat{X}^l_t:=\begin{pmatrix}
		\hat{M}_t^f \\ \mathbb{E}[X_t^l] \\ X_t^l-\mathbb{E}[X_t^l]
	\end{pmatrix},\quad \hat{\alpha}_t^l:=\begin{pmatrix}
		\mathbb{E}[\alpha_t^l] \\ \alpha_t^l-\mathbb{E}[\alpha_t^l]
	\end{pmatrix},\\
	\hat{A}_t^l:=\begin{pmatrix}
		\hat{A}_t^f+c\hat{C}_t^f+\hat{B}_t^f\hat{P}_t^f & 0 & 0 \\
		C_t^l & A_t^l & 0 \\
		0 & 0 & A_t^l
	\end{pmatrix},\quad \hat{B}_t^l:=\begin{pmatrix}
		\hat{B}_t^f \\ 0 \\ 0
	\end{pmatrix},\quad \hat{C}_t^l:=\begin{pmatrix}
		0 & 0 \\ B_t^l & 0 \\ 0 & B_t^l
	\end{pmatrix},\\
	\hat{D}_t^l:=\begin{pmatrix}
		0 & 0 & 0 \\
		0 & 0 & 0 \\
		F_t^l & D_t^l & D_t^l
	\end{pmatrix},\quad \hat{E}_t^l:=\begin{pmatrix}
		0 & 0 \\ 0 & 0 \\ E_t^l & E_t^l
	\end{pmatrix},\\
	\hat{F}_t^l:=\hat{H}_t^f-\hat{P}_t^f\hat{B}_t^f,\quad \hat{I}_t^l=\begin{pmatrix}
			0 & -Q_t^{f,13} & 0
	\end{pmatrix},\quad \hat{H}^l:=\begin{pmatrix}
		0 \\ G^{f,13} \\ 0
	\end{pmatrix},\\
	\hat{x}_0^l:=\begin{pmatrix}
		\int_I x_0^{f,u}\lambda(\mathrm{d}u) \\ \mathbb{E}[x_0^l] \\ x_0^l-\mathbb{E}[x_0^l]
	\end{pmatrix},\quad \hat{b}_t^l=\begin{pmatrix}
		\hat{b}_t^f \\ \mathbb{E}[b_t^l] \\ b_t^l-\mathbb{E}[b_t^l]
	\end{pmatrix},\quad \hat{\sigma}_t^l:=\begin{pmatrix}
		0 \\ 0 \\ \sigma_t^l
	\end{pmatrix},\quad \hat{g}_t^l:=\hat{g}_t^f-\hat{P}_t^f\hat{b}_t^f.
\end{gather*}
Then we can write
\begin{equation}\label{leader's state equation-augmented}
	\begin{cases}
		\mathrm{d}\hat{X}_t^l=\big(\hat{A}_t^l\hat{X}_t^l+\hat{B}_t^l\hat{N}_t^f+\hat{C}_t^l\hat{\alpha}_t^l+\hat{b}_t^l\big)\mathrm{d}t+\big(\hat{D}_t^l\hat{X}_t^l+\hat{E}_t^l\hat{\alpha}_t^l+\hat{\sigma}_t^l\big)\mathrm{d}W_t^l,\\
		\mathrm{d}\hat{N}_t^f=\big(\hat{F}_t^l\hat{N}_t^f+\hat{I}_t^l\hat{X}_t^l+\hat{g}_t^l\big)\mathrm{d}t,\\
		\hat{X}_0^l=\hat{x}_0^l,\quad \hat{N}_T^f=\hat{H}^l\hat{X}_T^l.
	\end{cases}
\end{equation}

Rewriting
\begin{equation*}
	U_t^l:=\begin{pmatrix}
		0 & 1 & 1 \\ 1 & 0 & 0
	\end{pmatrix}\hat{X}_t^l,\quad \alpha_t^l:=\begin{pmatrix}
		1 & 1
	\end{pmatrix}\hat{\alpha}_t^l,
\end{equation*}
the cost functional of the leader becomes
\begin{equation}\label{cost of leader-augmented}
	J^l(\hat{\alpha}^l;\hat{\alpha}^f)=\frac{1}{2}\mathbb{E}\biggl[ \int_0^T \Big\{ \hat{X}_t^{l\top}\hat{Q}_t^l\hat{X}_t^l+\hat{\alpha}_t^{l\top}\hat{R}_t^l\hat{\alpha}_t^l \Big\}\mathrm{d}t+\hat{X}_T^{l\top}\hat{G}^l\hat{X}_T^l \biggr],
\end{equation}
where
\begin{gather*}
		\hat{Q}_t^l:=\begin{pmatrix}
			0 & 1 & 1 \\ 1 & 0 & 0
		\end{pmatrix}^\top Q_t^l\begin{pmatrix}
			0 & 1 & 1 \\ 1 & 0 & 0
		\end{pmatrix},\quad
		\hat{R}_t^l:=  \begin{pmatrix}
				R_t^l & 0 \\
				0 & R_t^l
		\end{pmatrix},  \\
		\hat{G}^l:=\begin{pmatrix}
			0 & 1 & 1 \\ 1 & 0 & 0
		\end{pmatrix}^\top G^l\begin{pmatrix}
			0 & 1 & 1 \\ 1 & 0 & 0
		\end{pmatrix}.
\end{gather*}

The leader's problem is to minimize (\ref{cost of leader-augmented}) subject to the system equation (\ref{leader's state equation-augmented}).

\begin{dfn}\label{leader's problem}
	The leader's problem is said to be finite at $x_0^l\in\mathbb{R}^{n_2}$ if there exists $\alpha^l\in\mathcal{U}^l$ such that
	\begin{equation}\label{finite}
		J^l(\hat{\alpha}^l;\hat{\alpha}^f)>-\infty.
	\end{equation}
	It is said to be solvable at $x_0^l\in\mathbb{R}^{n_2}$ if there exists $\alpha^l\in\mathcal{U}^l$ such that
	\begin{equation}\label{solvable}
		J^l(\hat{\alpha}^l;\hat{\alpha}^f)=\inf_{\beta\in\mathcal{U}^l}J^l(\hat{\beta};\hat{\alpha}^f).
	\end{equation}
\end{dfn}

Since $\hat{Q}_t^l\ge 0$, $\hat{R}_t^l> 0$, and $\hat{G}^l\ge 0$, the leader's problem is always finite. Furthermore, we have the following result.

\begin{thm}\label{maximum principle leader problem}
	Assume conditions (A1), (A2), (A3), and (A4) hold. If the system
	\begin{equation}\label{FBSDEs leader}
		\begin{cases}
			\mathrm{d}\hat{\hat{X}}_t^l=\big(\hat{A}_t^l\hat{\hat{X}}_t^l+\hat{B}_t^l\hat{\hat{N}}_t^l+\hat{C}_t^l\hat{\alpha}_t^l+\hat{b}_t^l\big)\mathrm{d}t+\big(\hat{D}_t^l\hat{\hat{X}}_t^l+\hat{E}_t^l\hat{\alpha}_t^l+\hat{\sigma}_t^l\big)\mathrm{d}W_t^l,\\
			\mathrm{d}\hat{\hat{N}}_t^l=\big(\hat{F}_t^l\hat{\hat{N}}_t^l+\hat{I}_t^l\hat{X}_t^l+\hat{g}_t^l\big)\mathrm{d}t,\\
			\mathrm{d}y_t^l=-\big(\hat{A}_t^{l\top}y_t^l+\hat{D}_t^{l\top}z_t^l+\hat{Q}_t^l\hat{\hat{X}}_t^l+\hat{I}_t^{l\top}\psi_t^l\big)\mathrm{d}t+z_t^l\mathrm{d}W_t^l,\\
			\mathrm{d}\psi_t^l=-\big(\hat{F}_t^{l\top}\psi_t^l+\hat{B}_t^{l\top}y_t^l\big)\mathrm{d}t,\\
			\hat{\hat{X}}_0^l=\hat{x}_0^l,\quad \hat{\hat{N}}_T^l=\hat{H}^l\hat{\hat{X}}_T^l,\quad y_T^l=\hat{G}^l\hat{\hat{X}}_T^l-\hat{H}^{l\top}\psi_T^l,\quad \psi_0^l=0,
		\end{cases}
	\end{equation}
	admits a unique solution $(\hat{\hat{X}}^l,\hat{\hat{N}}^l,y^l,z^l,\psi^l)$, then the leader's problem is solvable at $x_0^l\in\mathbb{R}^{n_2}$ (denote the optimal pair by $(\hat{\hat{\alpha}}^l,\hat{\hat{X}}^l,\hat{\hat{N}}^l)$), and we have
	\begin{equation}\label{optimal control leader}
		\hat{R}_t^l\hat{\hat{\alpha}}_t^l+\hat{C}_t^{l\top}y_t^l+\hat{E}_t^{l\top}z_t^l=0.
	\end{equation}
\end{thm}

\begin{proof}
	Consider the system
	\begin{equation*}
		\begin{cases}
			\mathrm{d}\hat{X}_t^{0}=\big(\hat{A}_t^l\hat{X}_t^0+\hat{B}_t^l\hat{N}_t^0+\hat{C}_t^l\hat{\alpha}_t^l\big)\mathrm{d}t+\big(\hat{D}_t^l\hat{X}_t^0+\hat{E}_t^l\hat{\alpha}_t^l\big)\mathrm{d}W_t^l,\\
			\mathrm{d}\hat{N}_t^0=\big(\hat{F}_t^l\hat{N}_t^0+\hat{I}_t^l\hat{X}_t^0\big)\mathrm{d}t,\\
			\mathrm{d}y_t^0=-\big(\hat{A}_t^{l\top}y_t^0+\hat{D}_t^{l\top}z_t^0+\hat{Q}_t^l\hat{X}_t^0+\hat{I}_t^{l\top}\psi_t^0\big)\mathrm{d}t+z_t^0\mathrm{d}W_t^l,\\
			\mathrm{d}\psi_t^0=-\big(\hat{F}_t^{l\top}\psi_t^0+\hat{B}_t^{l\top}y_t^0\big)\mathrm{d}t,\\
			\hat{X}_0^0=0,\quad \hat{N}_T^0=\hat{H}^l\hat{X}_T^0,\quad y_T^0=\hat{G}^l\hat{X}_T^0-\hat{H}^{l\top}\psi_T^0,\quad \psi_0^0=0,
		\end{cases}
	\end{equation*}
	which admits a unique solution. Applying It\^o's formula yields
	\begin{multline*}
		\mathbb{E}[\hat{X}_T^{0\top}G^l\hat{X}_T^0]=\mathbb{E}[\hat{X}_T^{0\top}y_T^0]-\mathbb{E}[\hat{X}_0^{0\top}y_0^0]+\mathbb{E}[\hat{N}_T^{0\top}\psi_T^0]-\mathbb{E}[\hat{N}_0^{0\top}\psi_0^0]\\
		=\mathbb{E}\biggl[ \int_0^T\Big\{ -\hat{X}_t^{0\top}\hat{Q}_t^l\hat{X}_t^0+\hat{\alpha}_t^{l\top}\hat{C}_t^{l\top}y_t^0+\hat{\alpha}_t^{l\top}\hat{E}_t^{l\top}z_t^0 \Big\}\mathrm{d}t \biggr],
	\end{multline*}
	and
	\begin{multline*}
		\mathbb{E}[\hat{X}_T^{0\top}G^l\hat{\hat{X}}_T^l]=\mathbb{E}[\hat{X}_T^{0\top}y_T^l]-\mathbb{E}[\hat{X}_0^{0\top}y_0^l]+\mathbb{E}[\hat{N}_T^{0\top}\psi_T^l]-\mathbb{E}[\hat{N}_0^{0\top}\psi_0^l]\\
		=\mathbb{E}\biggl[ \int_0^T\Big\{ -\hat{X}_t^{0\top}\hat{Q}_t^l\hat{\hat{X}}_t^l+\hat{\alpha}_t^{l\top}\hat{C}_t^{l\top}y_t^l+\hat{\alpha}_t^{l\top}\hat{E}_t^{l\top}z_t^l \Big\}\mathrm{d}t \biggr].
	\end{multline*}
	For any $\beta\in\mathcal{U}^l$, we get
	\begin{equation}\label{function of lambda}
		\begin{aligned}
			&\hat{J}^l(\hat{\hat{\alpha}}^l+\lambda\hat{\beta})-\hat{J}^l(\hat{\hat{\alpha}}^l)\\
            &=\frac{1}{2}\mathbb{E}\biggl[ \int_0^T\Big\{ \big(\hat{\hat{X}}_t^l+\lambda\hat{X}_t^0\big)^\top\hat{Q}_t^l\big(\hat{\hat{X}}_t^l+\lambda\hat{X}_t^0\big)+\big(\hat{\hat{\alpha}}^l+\lambda\hat{\beta}\big)^\top\hat{R}_t^l\big(\hat{\hat{\alpha}}^l+\lambda\hat{\beta}\big) \Big\}\mathrm{d}t\\
			&\qquad\qquad +\big(\hat{\hat{X}}_T^l+\lambda\hat{X}_T^0\big)^\top\hat{G}^l\big(\hat{\hat{X}}_T^l+\lambda\hat{X}_T^0\big) \biggr]\\
            &\quad-\frac{1}{2}\mathbb{E}\biggl[ \int_0^T\Big\{ \hat{\hat{X}}_t^{l\top}\hat{Q}_t^l\hat{\hat{X}}_t^{l}+\hat{\hat{\alpha}}^{l\top}_t\hat{R}_t^l\hat{\hat{\alpha}}_t^{l} \Big\}\mathrm{d}t+\hat{\hat{X}}_T^{l\top}\hat{G}^l\hat{\hat{X}}_T^{l} \biggr]\\
			&=\frac{1}{2}\lambda^2\mathbb{E}\biggl[ \int_0^T\Big\{ \hat{X}_t^{0\top}\hat{Q}_t\hat{X}_t^0+\hat{\beta}_t^\top\hat{R}_t^l\hat{\beta}_t \Big\}\mathrm{d}t+\hat{X}_T^{0\top}\hat{G}^l\hat{X}_T^0 \biggr]\\
			&\quad +\lambda \mathbb{E}\biggl[ \int_0^T\Big\{ \hat{X}_t^{0\top}\hat{Q}_t\hat{\hat{X}}_t^l+\hat{\beta}_t^\top\hat{R}_t^l\hat{\hat{\alpha}}_t^l \Big\}\mathrm{d}t+\hat{X}_T^{0\top}\hat{G}^l\hat{\hat{X}}_T^l \biggr]\\
			&=\frac{1}{2}\lambda^2\mathbb{E}\biggl[ \int_0^T\Big\{ \hat{X}_t^{0\top}\hat{Q}_t\hat{X}_t^0+\hat{\beta}_t^\top\hat{R}_t^l\hat{\beta}_t \Big\}\mathrm{d}t+\hat{X}_T^{0\top}\hat{G}^l\hat{X}_T^0 \biggr]\\
			&\quad +\lambda \mathbb{E}\biggl[ \int_0^T\Big\{ \beta_t^\top\big(\hat{R}_t^l\hat{\hat{\alpha}}_t^l+\hat{C}_t^{l\top}y_t^l+\hat{E}_t^{l\top}z_t^l\big) \Big\}\mathrm{d}t \biggr].
		\end{aligned}
	\end{equation}
	Then $\hat{\hat{\alpha}}^l$ is optimal if and only if for every $\beta\in\mathcal{U}^l$, the function in $\lambda$ given by (\ref{function of lambda}) attains its minimum at $\lambda=0$, which is equivalent to (\ref{optimal control leader}).
\end{proof}

It remains to solve (\ref{FBSDEs leader}) and (\ref{optimal control leader}). From (\ref{optimal control leader}), we have
\begin{equation}\label{optimal control leader-by adjoint variables}
	\hat{\hat{\alpha}}_t^l=-(\hat{R}_t^l)^{-1}\big(\hat{C}_t^{l\top}y_t^l+\hat{E}_t^{l\top}z_t^l\big).
\end{equation}
Substituting (\ref{optimal control leader-by adjoint variables}) into (\ref{FBSDEs leader}) and defining
\begin{gather*}
	\tilde{X}_t:=\begin{pmatrix}
		\hat{\hat{X}}_t^l \\ -\psi_t^l
	\end{pmatrix},\quad \tilde{Y}_t:=\begin{pmatrix}
		y_t^l \\ \hat{\hat{N}}_t^l
	\end{pmatrix},\quad \tilde{Z}_t:=\begin{pmatrix}
		z_t^l \\ 0
	\end{pmatrix},\\
	\tilde{A}_t:=\begin{pmatrix}
		\hat{A}_t^l & 0 \\
		0 & -\hat{F}_t^{l\top}
	\end{pmatrix},\quad \tilde{B}_t:=\begin{pmatrix}
		-\hat{C}_t^l(\hat{R}_t^l)^{-1}\hat{C}_t^{l\top} & \hat{B}_t^l \\
		\hat{B}_t^{l\top} & 0
	\end{pmatrix},\quad \tilde{C}_t:=\begin{pmatrix}
		-\hat{C}_t^l(\hat{R}_t^l)^{-1}\hat{E}_t^{l\top} & 0 \\ 0 & 0
	\end{pmatrix},\\
	\tilde{D}_t:=\begin{pmatrix}
		\hat{D}_t^l & 0 \\ 0 & 0
	\end{pmatrix},\quad \tilde{E}_t:=\begin{pmatrix}
		-\hat{E}_t^l(R_t^l)^{-1}\hat{E}_t^{l\top} & 0 \\ 0 & 0
	\end{pmatrix},\quad \tilde{b}_t:=\begin{pmatrix}
		\hat{b}_t^l \\ 0
	\end{pmatrix},\quad \tilde{\sigma}_t:=\begin{pmatrix}
		\hat{\sigma}_t^l \\ 0
	\end{pmatrix},\\
	\tilde{Q}_t:=\begin{pmatrix}
		-\hat{Q}_t^l & \hat{I}_t^{l\top} \\ \hat{I}_t^l & 0
	\end{pmatrix},\quad \tilde{g}_t:=\begin{pmatrix}
		0 \\ \hat{g}_t^l
	\end{pmatrix},\quad \tilde{x}_0:=\begin{pmatrix}
		\hat{x}_0^l \\ 0
	\end{pmatrix},\quad \tilde{F}:=\begin{pmatrix}
		\hat{G}^l & \hat{H}^{l\top} \\
		\hat{H}^l & 0
	\end{pmatrix},\\
	\tilde{H}_t:=\begin{pmatrix}
		-(\hat{R}_t^l)^{-1}\hat{C}_t^{l\top} & 0
	\end{pmatrix},\quad \tilde{I}_t:=\begin{pmatrix}
		-(\hat{R}_t^l)^{-1}\hat{E}_t^{l\top} & 0
	\end{pmatrix},
\end{gather*}
we obtain the FBSDE
\begin{equation}\label{FBSDE leader-augmented}
	\begin{cases}
		\mathrm{d}\tilde{X}_t=\big(\tilde{A}_t\tilde{X}_t+\tilde{B}_t\tilde{Y}_t+\tilde{C}_t\tilde{Z}_t+\tilde{b}_t\big)\mathrm{d}t+\big(\tilde{D}_t\tilde{X}_t+\tilde{C}_t^\top\tilde{Y}_t+\tilde{E}_t\tilde{Z}_t+\tilde{\sigma}_t\big)\mathrm{d}W_t^l,\\
		\mathrm{d}\tilde{Y}_t=\big(\tilde{Q}_t\tilde{X}_t-\tilde{A}_t^\top\tilde{Y}_t-\tilde{D}_t^\top \tilde{Z}_t+\tilde{g}_t\big)\mathrm{d}t+\tilde{Z}_t\mathrm{d}W_t^l,\\
		\tilde{X}_0=\tilde{x}_0,\quad \tilde{Y}_T=\tilde{F}\tilde{X}_T,
	\end{cases}
\end{equation}
and the optimal control $\hat{\hat{\alpha}}^l$ in (\ref{optimal control leader-by adjoint variables}) can be written as
\begin{equation}\label{optimal control leader-by augmented variables}
	\hat{\hat{\alpha}}^l_t=\tilde{H}_t\tilde{Y}_t+\tilde{I}_t\tilde{Z}_t.
\end{equation}

By classical FBSDE theory, assuming $\tilde{Y}_t=\tilde{P}_t\tilde{X}_t+\tilde{\varphi}_t$, we obtain
\begin{equation}\label{tilde Z}
	\tilde{Z}_t=(I-\tilde{P}_t\tilde{E}_t)^{-1}\big[\tilde{P}_t\big(\tilde{D}_t+\tilde{C}_t^{\top}\tilde{P}_t\big)\tilde{X}_t+\tilde{P}_t\tilde{C}_t^\top\tilde{\varphi}_t+\tilde{P}_t\tilde{\sigma}_t\big],
\end{equation}
if $I-\tilde{P}_t\tilde{E}_t$ is invertible, the Riccati equation
\begin{equation}\label{Riccati leader-augmented}
	\begin{cases}
		\dot{\tilde{P}}_t+\tilde{P}_t\tilde{A}_t+\tilde{A}_t^\top\tilde{P}_t+\tilde{P}_t\tilde{B}_t\tilde{P}_t\\
		\qquad +\big(\tilde{P}_t\tilde{C}_t+\tilde{D}_t^{\top}\big)(I-\tilde{P}_t\tilde{E}_t)^{-1}\tilde{P}_t\big(\tilde{D}_t+\tilde{C}_t^{\top}\tilde{P}_t\big)-\tilde{Q}_t=0,\\
		\tilde{P}_T=\tilde{F},
	\end{cases}
\end{equation}
and the forward-backward system
\begin{equation}\label{FBSDE leader-final}\hspace{-3mm}
	\begin{cases}
		\mathrm{d}\tilde{X}_t=\big\{ \big[\tilde{A}_t+\tilde{B}_t\tilde{P}_t+\tilde{C}_t(I-\tilde{P}_t\tilde{E}_t)^{-1}\tilde{P}_t\big(\tilde{D}_t+\tilde{C}_t^\top\tilde{P}_t\big)\big]\tilde{X}_t\\
		\qquad\qquad +\big[\tilde{C}_t(I-\tilde{P}_t\tilde{E}_t)^{-1}\tilde{P}_t\tilde{C}_t^\top+\tilde{B}_t\big]\tilde{\varphi}_t+\big[\tilde{b}_t+\tilde{C}_t(I-\tilde{P}_t\tilde{E}_t)^{-1}\tilde{P}_t\tilde{\sigma}_t\big] \big\}\mathrm{d}t\\
		\qquad\quad +\big\{ \big[\tilde{D}_t+\tilde{C}_t^\top\tilde{P}_t+\tilde{E}_t(I-\tilde{P}_t\tilde{E}_t)^{-1}\tilde{P}_t\big(\tilde{D}_t+\tilde{C}_t^\top\tilde{P}_t\big)\big]\tilde{X}_t\\
		\qquad\qquad +\big[\tilde{E}_t(I-\tilde{P}_t\tilde{E}_t)^{-1}\tilde{P}_t\tilde{C}_t^\top+\tilde{C}_t^\top\big]\tilde{\varphi}_t+\big[\tilde{\sigma}_t+\tilde{E}_t(I-\tilde{P}_t\tilde{E}_t)^{-1}\tilde{P}_t\tilde{\sigma}_t\big] \big\}\mathrm{d}W_t^l,\\
		\mathrm{d}\tilde{\varphi}_t+\big\{ \big[\tilde{P}_t\tilde{B}_t+\hat{A}_t^\top+\big(\tilde{P}_t\tilde{C}_t+\tilde{D}_t^\top\big)(I-\tilde{P}_t\tilde{E}_t)^{-1}\tilde{P}_t\tilde{C}_t^\top\big]\tilde{\varphi}_t\\
		\qquad\qquad +\big[\tilde{P}_t\tilde{b}_t-\tilde{g}_t+\big(\tilde{P}_t\tilde{C}_t+\tilde{D}_t^\top\big)(I-\tilde{P}_t\tilde{E}_t)^{-1}\tilde{P}_t\tilde{\sigma}_t\big] \big\}\mathrm{d}t=0,\\
		\tilde{X}_0=\tilde{x}_0,\quad \tilde{\varphi}_T=0.
	\end{cases}
\end{equation}

Equation (\ref{FBSDE leader-final}) constitutes a decoupled forward-backward structure, where the backward component is a boundary value problem for an ODE. This structure allows for sequentially solving the backward part first, followed by the forward part, to obtain the adapted solution of the equation.

\begin{thm}\label{solution leader problem}
	Assume conditions (A1), (A2), (A3), and (A4) hold. If the Riccati equation (\ref{Riccati leader-augmented}) and the forward-backward system (\ref{FBSDE leader-final}) admit unique solutions $\tilde{P}$ and $(\tilde{X},\tilde{\varphi})$ respectively, $I-\tilde{P}_t\tilde{E}_t$ is invertible, then the leader's problem has a unique optimal control $\hat{\hat{\alpha}}^l$, given by
	\begin{multline}\label{optimal control leader-feedback}
		\hat{\hat{\alpha}}_t^l=\big[\tilde{H}_t\tilde{P}_t+\tilde{I}_t(I-\tilde{P}_t\tilde{E}_t)^{-1}\tilde{P}_t\big(\tilde{D}_t+\tilde{C}_t^\top\tilde{P}_t\big)\big]\tilde{X}_t\\
		+\big[\tilde{H}_t+\tilde{I}_t(I-\tilde{P}_t\tilde{E}_t)^{-1}\tilde{P}_t\tilde{C}_t^{\top}\big]\tilde{\varphi}_t+\tilde{I}_t(I-\tilde{P}_t\tilde{E}_t)^{-1}\tilde{P}_t\tilde{\sigma}_t.
	\end{multline}
\end{thm} 

	\subsection{Stackelberg-Nash equilibrium}

By combining Theorems \ref{solution followers problem} and \ref{solution leader problem}, we obtain the following result.

\begin{thm}\label{solution of Stackelberg games}
	Assume that conditions (A1), (A2), (A3), and (A4) hold. Let $I+P_t^{f,u}E_t^f(R_t^f)^{-1}E_t^{f\top}$ and $I-\tilde{P}_t\tilde{E}_t$ be invertible. If the Riccati equations (\ref{Riccati leader}) and (\ref{Riccati leader-augmented}) and the forward-backward system (\ref{FBSDE leader-final}) admit unique solutions $\hat{P}^f$, $\tilde{P}$ and $(\tilde{X},\tilde{\varphi})$ respectively, then the problem admits a Stackelberg-Nash equilibrium $(\hat{\hat{\alpha}}^l,\hat{\hat{\alpha}}^f)$ given by 
	\begin{equation}\label{Stackelberg-Nash equilibrium}
		\begin{aligned}
			\hat{\hat{\alpha}}_t^l&=\big[\tilde{H}_t\tilde{P}_t+\tilde{I}_t(I-\tilde{P}_t\tilde{E}_t)^{-1}\tilde{P}_t\big(\tilde{D}_t+\tilde{C}_t^\top\tilde{P}_t\big)\big]\tilde{X}_t\\
			&\quad +\big[\tilde{H}_t+\tilde{I}_t(I-\tilde{P}_t\tilde{E}_t)^{-1}\tilde{P}_t\tilde{C}_t^{\top}\big]\tilde{\varphi}_t+\tilde{I}_t(I-\tilde{P}_t\tilde{E}_t)^{-1}\tilde{P}_t\tilde{\sigma}_t,\\
			\hat{\hat{\alpha}}_t^f&=-(R_t^f)^{-1}\big[B_t^{f\top}P_t^{f}+E_t^{f\top}(\hat R_t^f)^{-1}P_t^f(D_t^f-E_t^f(R_t^f)^{-1}B_t^{f\top}P_t^f)\big]\hat{\hat{X}}_t^{f,u}\\
			&\quad -(R_t^f)^{-1}\big[B_t^{f\top}-E_t^{f\top}(\hat R_t^f)^{-1}P_t^fE_t^f(R_t^f)^{-1}B_t^{f\top}\big]\hat{\hat{\varphi}}_t^{f,u}\\
			&\quad -(R_t^f)^{-1}E_t^{f\top}(\hat R_t^f)^{-1}P_t^f\big[F_t^fG\hat{\hat{X}}_t^{f,u}+\sigma_t^f\big],
		\end{aligned}
	\end{equation}
	where $(\hat{\hat{X}}^f,\hat{\hat{\varphi}}^f)$ is the unique solution of the following forward-backward system
	\begin{equation}\label{optimality system}
		\begin{cases}
			\mathrm{d}\hat{\hat{X}}_t^{f,u}=\big(\hat{A}_t^{f}\hat{\hat{X}}_t^{f,u}+\hat{B}_t^{f}\hat{\hat{\varphi}}_t^{f,u}+\hat{C}_t^{f}G\hat{\hat{X}}_t^{f,u}+\hat{b}_t^{f}\big)\mathrm{d}t\\
			\qquad\qquad +\big(\hat{D}_t^{f}\hat{\hat{X}}_t^{f,u}+\hat{E}_t^{f}\hat{\hat{\varphi}}_t^{f,u}+\hat{F}_t^{f}G\hat{\hat{X}}_t^{f,u}+\hat{\sigma}_t^{f}\big)\mathrm{d}W_t^{f,u},\\
			\mathrm{d}\hat{\hat{\varphi}}_t^{f,u}=\big(\hat{H}_t^{f}\hat{\hat{\varphi}}_t^{f,u}+\hat{I}_t^{f}G\hat{\hat{X}}_t^{f,u}+\hat{g}_t^{f}\big)\mathrm{d}t,\\
			\hat{\hat{X}}_0^{f,u}=x_0^{f,u},\quad \hat{\hat{\varphi}}_T^{f,u}=G^{f,12}G\hat{\hat{X}}_T^{f,u}+G^{f,13}X_T^*,
		\end{cases}
	\end{equation}
	and $X^*:=\begin{pmatrix}
		0 & I_{n_2} & 0
	\end{pmatrix}\hat{\hat{X}}^l$, $\hat{\hat{X}}^l:=\begin{pmatrix}
	I_{n_1+2n_2} & 0
	\end{pmatrix}\tilde{X}$.
\end{thm} 

	\section{Existence and uniqueness of solutions to graphon-aggregated FBSDEs}
	
	\subsection{Existence and uniqueness of solutions to graphon-aggregated FBSDEs}

In the followers' problem of the Stackelberg game, the existence and uniqueness of the followers' Nash equilibrium relies on the existence and uniqueness of solutions to the graphon-aggregated FBSDE (\ref{graphon-aggregated FBSDE followers}). In this subsection, we generalize the continuity method from classical FBSDE theory (\cite{Hu-Peng-95, Yong-97, Peng-Wu-99}) to discuss the existence and uniqueness of solutions to linear graphon-aggregated FBSDEs and their estimates.

Consider the general linear graphon-aggregated FBSDE:
\begin{equation}\label{GMFLFBSDE}
	\begin{cases}
		\mathrm{d}X_t^u=\big(A_t^{21,u}X_t^u+A_t^{22,u}Y_t^u+A_t^{23,u}Z_t^u+B_t^{2,u}GX_t^u+b_t^u\big)\mathrm{d}t,\\
		\qquad\qquad +\big(A_t^{31,u}X_t^u+A_t^{32,u}Y_t^u+A_t^{33,u}Z_t^u+B_t^{3,u}GX_t^u+\sigma_t^u\big)\mathrm{d}W_t^u,\\
		\mathrm{d}Y_t^u=\big(A_t^{11,u}X_t^u+A_t^{12,u}Y_t^u+A_t^{13,u}Z_t^u+B_t^{1,u}GX_t^u+g_t^u\big)\mathrm{d}t+Z_t^u\mathrm{d}W_t^u,\\
		X_0^u=x_0^u,\quad Y_T^u=G^{1,u}X_T^u+G^{2,u}GX_T^u+h^u,\quad u\in I,
	\end{cases}
\end{equation}
where $X,Y,Z$ are $\mathbb{R}^n$-valued processes, $G\in\mathcal{W}_0$, and
\begin{equation*}
	GX^u:=\int_I G(u,v)X^v\lambda(\mathrm{d}v),\quad u\in I.
\end{equation*}
Here, $W\equiv(W_t^u)_{u\in I,t\in[0,T]}$ is a family of essentially pairwise independent Brownian motions, where for each $u\in I$, $W^u$ is a one-dimensional standard Brownian motion. The filtration is defined by
\[
\mathcal{F}_t^u:=\sigma\{x_0^u,W_s^u,0\le s\le t\}\lor\mathcal{N},
\]
with $\mathcal{N}$ denoting the collection of all $P$-null sets.

\begin{Ass}[S1]
	For any $i,j\in\{1,2,3\}$, $A^{ij}\in L_\boxtimes^\infty(0,T;\mathbb{R}^n)$, the processes $b,\sigma,g\in L_{\boxtimes}^2(0,T;\mathbb{R}^n)$, the terminal condition $h\in L_{\boxtimes_T}^2(\mathbb{R}^n)$.
\end{Ass}

\begin{Ass}[S2]
	There exists a constant $K_1>0$ such that for any $t\in[0,T]$ and $x,y,z\in L^2_{\mathcal{I}}(\mathbb{R}^n)$, the following inequality holds:
	\begin{multline}\label{monotonicity condition-1}
		\int_I\left\{\begin{pmatrix}
			x^u \\ y^u \\ z^u
		\end{pmatrix}^\top\begin{pmatrix}
			A_t^{11,u} & A_t^{12,u} & A_t^{13,u} \\
			A_t^{21,u} & A_t^{22,u} & A_t^{23,u} \\
			A_t^{31,u} & A_t^{32,u} & A_t^{33,u}
		\end{pmatrix}\begin{pmatrix}
			x^u \\ y^u \\ z^u
		\end{pmatrix}+\begin{pmatrix}
			x^u \\ y^u \\ z^u
		\end{pmatrix}^\top\begin{pmatrix}
			B_t^{1,u} \\ B_t^{2,u} \\ B_t^{3,u}
		\end{pmatrix}Gx^u\right\}\lambda(\mathrm{d}u)\\
		\le -K_1\int_I\big\{ |x^u|^2+|y^u|^2+|z^u|^2 \big\}\lambda(\mathrm{d}u).
	\end{multline}
	Additionally, for any $x\in L_{\mathcal{I}}^2(\mathbb{R}^n)$,
	\begin{equation}\label{monotonicity condition-2}
		\int_I x^{u\top}G^{1,u}x^u\lambda(\mathrm{d}u)\ge 0,\quad \int_I x^{u\top}G^{2,u}Gx^{u}\lambda(\mathrm{d}u)\ge 0.
	\end{equation}
\end{Ass}

\begin{thm}\label{existence and uniqueness of GMFLFBSDE}
	Assume that conditions (S1) and (S2) hold. Then the linear graphon-aggregated FBSDE (\ref{GMFLFBSDE}) admits a unique solution, and the following estimate is satisfied:
	\begin{equation}\label{estimate of GMFLFBSDE}
        \begin{aligned}
		&\mathbb{E}^\boxtimes\biggl[ \sup_{0\le t\le T}|X_t|^2+\sup_{0\le t\le T}|Y_t|^2+\int_0^T|Z_t|^2\mathrm{d}t \biggr]\\
        &\qquad \le K\mathbb{E}^\boxtimes\biggl[ |x_0|^2+|h|^2+\int_0^T\{|b_t|^2+|\sigma_t|^2+|g_t|^2\}\mathrm{d}t \biggr],
	    \end{aligned}
    \end{equation}
for some constant $K>0$.
\end{thm}

Before proving Theorem \ref{existence and uniqueness of GMFLFBSDE}, we establish several lemmas.

\begin{lemma}\label{uniqueness of GMFLFBSDE}
	Under assumptions (S1) and (S2), if a solution to equation (\ref{GMFLFBSDE}) exists, it is unique.
\end{lemma}

\begin{proof}
	Assume, for contradiction, that the solution is not unique. Let $(X^1,Y^1,Z^1)$ and $(X^2,Y^2,Z^2)$ be two distinct solutions, and denote their difference by $(\varDelta X,\varDelta Y,\varDelta Z)$. Applying Itô's formula to $\mathrm{d}(\varDelta X_t^{u\top}\varDelta Y_t^u)$ and using condition (\ref{monotonicity condition-1}), integration yields
	\begin{multline*}
		\mathbb{E}^\boxtimes\big[\varDelta X_T^\top\varDelta Y_T\big]\\
		=\mathbb{E}^\boxtimes\left[ \int_0^T \left\{\begin{pmatrix}
			\varDelta X_t \\ \varDelta Y_t \\ \varDelta Z_t
		\end{pmatrix}^\top\begin{pmatrix}
			A_t^{11} & A_t^{12} & A_t^{13} \\
			A_t^{21} & A_t^{22} & A_t^{23} \\
			A_t^{31} & A_t^{32} & A_t^{33}
		\end{pmatrix}\begin{pmatrix}
			\varDelta X_t \\ \varDelta Y_t \\ \varDelta Z_t
		\end{pmatrix}+\begin{pmatrix}
			\varDelta X_t \\ \varDelta Y_t \\ \varDelta Z_t
		\end{pmatrix}^\top\begin{pmatrix}
			B_t^{1} \\ B_t^{2} \\ B_t^{3}
		\end{pmatrix}G\varDelta X_t\right\}\mathrm{d}t \right]\\
		\le -K_1\mathbb{E}^\boxtimes\biggl[ \int_0^T \{|\varDelta X_t|^2+|\varDelta Y_t|^2+|\varDelta Z_t|^2\}\mathrm{d}t \biggr].
	\end{multline*}
	By condition (\ref{monotonicity condition-2}),
	\begin{equation*}
	    \begin{aligned}	
          0&\le \mathbb{E}^\boxtimes \big[\varDelta X_T^\top G^1 \varDelta X_T+\varDelta X_T^\top G^2 G\varDelta X_T\big]\\
           &\quad +K_1\mathbb{E}^\boxtimes\biggl[ \int_0^T\{|\varDelta X_t|^2+|\varDelta Y_t|^2+|\varDelta Z_t|^2\}\mathrm{d}t  \biggr]\le 0,
	    \end{aligned}
    \end{equation*}
	which implies
	\begin{equation*}
		(\varDelta X,\varDelta Y,\varDelta Z)=0,\quad P\boxtimes\lambda\text{-}a.e.\qedhere
	\end{equation*}
\end{proof}

To establish existence and derive estimates, consider the parameterized equation:
\begin{equation}\label{GMFLFBSDE parameterized alpha}
	\begin{cases}
		\mathrm{d}X_t^u=\big\{-(1-\alpha)K_1Y_t^u+\alpha\big(A_t^{21,u}X_t^u+A_t^{22,u}Y_t^u+A_t^{23,u}Z_t^u\\
        \qquad\qquad +B_t^{2,u}GX_t^u\big)+b_t^u\big\}\mathrm{d}t+\big\{-(1-\alpha)K_1Z_t^u\\
		\qquad\qquad +\alpha\big(A_t^{31,u}X_t^u+A_t^{32,u}Y_t^u+A_t^{33,u}Z_t^u+B_t^{3,u}GX_t^u\big)+\sigma_t^u\big\}\mathrm{d}W_t^u,\\
		\mathrm{d}Y_t^u=\big\{-(1-\alpha)K_1X_t^u+\alpha\big(A_t^{11,u}X_t^u+A_t^{12,u}Y_t^u+A_t^{13,u}Z_t^u\\
        \qquad\qquad +B_t^{1,u}GX_t^u\big)+g_t^u\big\}\mathrm{d}t+Z_t^u\mathrm{d}W_t^u,\\
		X_0^u=x_0^u,\quad Y_T^u=(1-\alpha)X_T^u+\alpha G^{1,u}X_T^u+\alpha G^{2,u}GX_T^u+h^u,\quad u\in I,
	\end{cases}
\end{equation}
where $\alpha\in [0,1]$, and $b,\sigma,g\in L_\boxtimes^2(0,T;\mathbb{R}^n)$.

\begin{lemma}\label{existence and uniqueness of GMFLFBSDE parameterized alpha=0}
	When $\alpha=0$, the equation (\ref{GMFLFBSDE parameterized alpha}) reduces to:
	\begin{equation}\label{GMFLFBSDE parameterized alpha=0}
		\begin{cases}
			\mathrm{d}X_t^u=(-K_1Y_t^u+b_t^u)\mathrm{d}t+(-K_1Z_t^u+\sigma_t^u)\mathrm{d}W_t^u,\\
			\mathrm{d}Y_t^u=(-K_1X_t^u+g_t^u)\mathrm{d}t+Z_t^u\mathrm{d}W_t^u,\\
			X_0^u=x_0^u,\quad Y_T^u=X_T^u+h^u,\quad u\in I.
		\end{cases}
	\end{equation}
	For any $b,\sigma,g\in L_{\boxtimes}^2(0,T;\mathbb{R}^n)$, (\ref{GMFLFBSDE parameterized alpha=0}) admits a unique solution $(X,Y,Z)\in \mathbb{S}^2_\boxtimes(\mathbb{R}^n)\times\mathbb{S}^2_{\boxtimes}(\mathbb{R}^n)\times\mathbb{H}^2_{\boxtimes}(\mathbb{R}^n)$, and the following estimate holds:
	\begin{equation}\label{estimate of GMFLFBSDE parameterized alpha=0}
        \begin{aligned}
		&\mathbb{E}^\boxtimes\biggl[ \sup_{0\le t\le T}|X_t|^2+\sup_{0\le t\le T}|Y_t|^2+\int_0^T |Z_t|^2\mathrm{d}t \biggr]\\
        &\quad \le K\mathbb{E}^\boxtimes\biggl[ |x_0|^2+|h|^2+ \int_0^T\{|b_t|^2+|\sigma_t|^2+|g_t|^2\}\mathrm{d}t \biggr].
	    \end{aligned}
    \end{equation}
\end{lemma}

\begin{proof}
	Define $\tilde{Y}_t^u:=Y_t^u-X_t^u$ and $\tilde{Z}_t^u:=(K_1+1)Z_t^u-\sigma_t^u$. By classical BSDE theory, the equation
	\begin{equation}\label{BSDE}
		\begin{cases}
			\mathrm{d}\tilde{Y}_t^u=(K_1\tilde{Y}_t^u+g_t^u-b_t^u)\mathrm{d}t+\tilde{Z}_t^u\mathrm{d}W_t^u,\\
			\tilde{Y}_T^u=h^u,\quad u\in I,
		\end{cases}
	\end{equation}
	has a unique solution $(Y,Z)\in\mathbb{S}_\boxtimes^2(\mathbb{R}^n)\times\mathbb{H}_\boxtimes^2(\mathbb{R}^n)$, with the estimate
	\begin{equation}\label{estimate of BSDE}
		\mathbb{E}^\boxtimes\biggl[ \sup_{0\le t\le T}|Y_t|^2+\int_0^T |Z_t|^2\mathrm{d}t \biggr]\le K\mathbb{E}^\boxtimes\biggl[ |h|^2+\int_0^T\{|b_t|^2+|g_t|^2\}\mathrm{d}t \biggr].
	\end{equation}
	Substituting (\ref{BSDE}) into the forward equation of (\ref{GMFLFBSDE parameterized alpha=0}) yields
	\begin{equation}\label{SDE}
		\begin{cases}
			\mathrm{d}X_t^u=(-K_1X_t^u-K_1\tilde{Y}_t^u+b_t^u)\mathrm{d}t+(-K_1Z_t^u+\sigma_t^u)\mathrm{d}W_t^u,\\
			X_0^u=x_0^u,\quad u\in I,
		\end{cases}
	\end{equation}
	which, by classical SDE theory, admits a unique solution $X\in\mathbb{S}_\boxtimes^2(\mathbb{R}^n)$, satisfying
	\begin{equation}\label{estimate of SDE}
		\mathbb{E}^\boxtimes\biggl[ \sup_{0\le t\le T}|X_t|^2 \biggr]\le K\mathbb{E}^\boxtimes\biggl[ |x_0|^2+\int_0^T\{|\tilde{Y}_t|^2+|\tilde{Z}_t|^2+|b_t|^2+|\sigma_t|^2\}\mathrm{d}t \biggr].
	\end{equation}
	Thus, $(X,Y,Z)$ exists and is unique, and inequality (\ref{estimate of GMFLFBSDE parameterized alpha=0}) follows from estimates (\ref{estimate of BSDE}) and (\ref{estimate of SDE}).
\end{proof}

\begin{lemma}\label{existence and uniqueness of SDE graphon aggregation}
	Suppose $b,\sigma:\varOmega\times I\times[0,T]\times\mathbb{R}^n\times\mathbb{R}^n\to\mathbb{R}^n$ satisfy
	\begin{gather*}
		b(\cdot,0,0),\sigma(\cdot,0,0)\in L_{\boxtimes}^2(0,T;\mathbb{R}^n),\\
		|b(t,x_1,y_1)-b(t,x_2,y_2)|\le L\big( |x_1-x_2|+|y_1-y_2| \big),\quad t\in[0,T],x_1,x_2,y_1,y_2\in\mathbb{R}^n,
	\end{gather*}
	and the SDE with graphon aggregation term
	\begin{equation}\label{SDE graphon aggregation}
		\begin{cases}
			\mathrm{d}X_t^u=b^u(t,X_t^u,GX_t^u)\mathrm{d}t+\sigma^u(t,X_t^u,GX_t^u)\mathrm{d}W_t^u,\quad u\in I,\\
			X_0^u=x_0^u,
		\end{cases}
	\end{equation}
	admits a unique solution $X\in\mathbb{H}_\boxtimes^2(\mathbb{R}^n)$. Then $X\in\mathbb{S}_\boxtimes^2(\mathbb{R}^n)$, and the following estimate holds:
	\begin{equation*}
		\mathbb{E}^\boxtimes\biggl[\sup_{0\le t\le T}|X_t|^2\biggr]\le K\mathbb{E}^\boxtimes\biggl[ |x_0|^2+\int_0^T \{|b(t,0,0)|^2+|\sigma(t,0,0)|^2\}\mathrm{d}t \biggr].
	\end{equation*}
\end{lemma}

\begin{proof}
	The proof follows similarly to Theorem 3.2.2, Chapter 3 of \cite{Zhang-17}, with Proposition \ref{extension of domain of graphon operator} used to handle estimates involving the graphon aggregation term.
\end{proof}

\begin{lemma}\label{estimate of GMFLFBSDE alpha delta}
	Assume conditions (S1) and (S2) hold. Suppose that for any $b,\sigma,g\in L_\boxtimes^2(0,T;\mathbb{R}^n)$ and $h\in L_{\boxtimes_T}^2(\mathbb{R}^n)$, the linear graphon-aggregated FBSDE (\ref{GMFLFBSDE parameterized alpha}) with parameter $\alpha_0$ admits a unique solution $(X,Y,Z)\in\mathbb{S}_\boxtimes^2(\mathbb{R}^n)\times\mathbb{S}_\boxtimes^2(\mathbb{R}^n)\times\mathbb{H}_\boxtimes^2(\mathbb{R}^n)$, and the estimate
	\begin{equation}\label{estimate of GMFLFBSDE parameterized alpha0}
        \begin{aligned}
		&\mathbb{E}^\boxtimes\biggl[ \sup_{0\le t\le T}|X_t|^2+\sup_{0\le t\le T}|Y_t|^2+\int_0^T |Z_t|^2\mathrm{d}t \biggr]\\
        &\quad \le K\mathbb{E}^\boxtimes\biggl[ |x_0|^2+|h|^2+ \int_0^T\{|b_t|^2+|\sigma_t|^2+|g_t|^2\}\mathrm{d}t \biggr]
	    \end{aligned}
    \end{equation}
	is satisfied. Then there exists a constant $\delta_0>0$, independent of $\alpha_0$, such that for any $\alpha\in[\alpha_0,\alpha_0+\delta_0]$, and any $b,\sigma,g\in L_\boxtimes^2(0,T;\mathbb{R}^n)$, $h\in L_{\boxtimes_T}^2(\mathbb{R}^n)$, (\ref{GMFLFBSDE parameterized alpha}) with parameter $\alpha$ admits a unique solution $(X,Y,Z)\in\mathbb{S}_\boxtimes^2(\mathbb{R}^n)\times\mathbb{S}_\boxtimes^2(\mathbb{R}^n)\times\mathbb{H}_\boxtimes^2(\mathbb{R}^n)$, and the estimate
	\begin{equation}\label{estimate of GMFLFBSDE parameterized alpha}
        \begin{aligned}
		&\mathbb{E}^\boxtimes\biggl[ \sup_{0\le t\le T}|X_t|^2+\sup_{0\le t\le T}|Y_t|^2+\int_0^T |Z_t|^2\mathrm{d}t \biggr]\\
        &\quad \le K\mathbb{E}^\boxtimes\biggl[ |x_0|^2+|h|^2+ \int_0^T\{|b_t|^2+|\sigma_t|^2+|g_t|^2\}\mathrm{d}t \biggr]
	    \end{aligned}    
    \end{equation}
	holds. Here, the constant $K$ in (\ref{estimate of GMFLFBSDE parameterized alpha0}) and (\ref{estimate of GMFLFBSDE parameterized alpha}) denotes a generic constant, not necessarily the same in each occurrence.
\end{lemma}

\begin{proof}
	Assume that for parameter $\alpha_0$, a unique solution to the linear graphon-aggregated FBSDE (\ref{GMFLFBSDE parameterized alpha}) exists. Consider the equation:
	\begin{equation}\label{GMFLFBSDE parameterized alpha delta}
		\begin{cases}
			\mathrm{d}X_t^u=\big\{-(1-\alpha)K_1Y_t^u+\alpha\big(A_t^{21,u}X_t^u+A_t^{22,u}Y_t^u+A_t^{23,u}Z_t^u+B_t^{2,u}GX_t^u\big)\\
			\qquad\qquad +\delta\big(K_1y_t^u+A_t^{21,u}x_t^u+A_t^{22,u}y_t^u+A_t^{23,u}z_t^u+B_t^{2,u}Gx_t^u\big)+b_t^u\big\}\mathrm{d}t,\\
			\qquad\quad +\big\{-(1-\alpha)K_1Z_t^u+\alpha\big(A_t^{31,u}X_t^u+A_t^{32,u}Y_t^u+A_t^{33,u}Z_t^u+B_t^{3,u}GX_t^u\big)\\
			\qquad\qquad +\delta\big(K_1z_t^u+A_t^{31,u}x_t^u+A_t^{32,u}y_t^u+A_t^{33,u}z_t^u+B_t^{3,u}Gx_t^u\big)+\sigma_t^u\big\}\mathrm{d}W_t^u,\\
			\mathrm{d}Y_t^u=\big\{-(1-\alpha)K_1X_t^u+\alpha\big(A_t^{11,u}X_t^u+A_t^{12,u}Y_t^u+A_t^{13,u}Z_t^u+B_t^{1,u}GX_t^u\big)\\
			\qquad\qquad +\delta\big(K_1x_t^u+A_t^{11,u}x_t^u+A_t^{12,u}y_t^u+A_t^{13,u}z_t^u+B_t^{1,u}Gx_t^u\big)+g_t^u\big\}\mathrm{d}t+Z_t^u\mathrm{d}W_t^u,\\
			X_0^u=x_0^u,\\
            Y_T^u=(1-\alpha)X_T^u+\alpha G^{1,u}X_T^u+\alpha G^{2,u}GX_T^u-\delta\big(x_T^u-G^{1,u}x_T^u-G^{2,u}Gx_T^u\big)+h^u.
		\end{cases}
	\end{equation}
	For given $(x,y,z)\in L_\boxtimes^2(\mathbb{R}^{3n})$, there exists a unique solution $(X,Y,Z)\in L_\boxtimes^2(\mathbb{R}^{3n})$ satisfying
	\begin{multline*}
		\mathbb{E}^\boxtimes\biggl[ \sup_{0\le t\le T}|X_t|^2+\sup_{0\le t\le T}|Y_t|^2+\int_0^T |Z_t|^2\mathrm{d}t \biggr]\\
		\le K\delta^2\mathbb{E}^\boxtimes\biggl[ \int_0^T \{|x_t|^2+|y_t|^2+|z_t|^2\}\mathrm{d}t \biggr]+K\mathbb{E}^\boxtimes\biggl[ |x_0|^2+|h|^2+\int_0^T \{|b_t|^2+|\sigma_t|^2+|g_t|^2\}\mathrm{d}t \biggr].
	\end{multline*}
	Define the mapping $\varPhi:\varPhi(x,y,z):=(X,Y,Z)$. If $\varPhi$ has a fixed point, then the equation (\ref{GMFLFBSDE parameterized alpha delta}) admits a unique solution, and for sufficiently small $\delta$, the estimate holds. We now show that for sufficiently small $\delta$, $\varPhi$ is a contraction.
	
	Apply Itô's formula to $\varDelta X_t^{u\top}\varDelta Y_t^u$, we get
	\begin{multline*}
		\mathrm{d}(\varDelta X_t^{u\top}\varDelta Y_t^u)=\Big\{-(1-\alpha)K\big(|\varDelta X_t^u|^2+|\varDelta Y_t^u|^2+|\varDelta Z_t^u|^2\big)\\
		+\alpha\begin{pmatrix}
			\varDelta X_t^u \\ \varDelta Y_t^u \\ \varDelta Z_t^u
		\end{pmatrix}^\top\begin{pmatrix}
			A_t^{11,u} & A_t^{12,u} & A_t^{13,u} \\
			A_t^{21,u} & A_t^{22,u} & A_t^{23,u} \\
			A_t^{31,u} & A_t^{32,u} & A_t^{33,u}
		\end{pmatrix}\begin{pmatrix}
			\varDelta X_t^u \\ \varDelta Y_t^u \\ \varDelta Z_t^u
		\end{pmatrix}+\alpha\begin{pmatrix}
			\varDelta X_t^u \\ \varDelta Y_t^u \\ \varDelta Z_t^u
		\end{pmatrix}^\top\begin{pmatrix}
			B_t^{1,u} \\ B_t^{2,u} \\ B_t^{3,u}
		\end{pmatrix}G\varDelta X_t^u\\
		+\delta\begin{pmatrix}
			\varDelta X_t^u \\ \varDelta Y_t^u \\ \varDelta Z_t^u
		\end{pmatrix}^\top\begin{pmatrix}
			B_t^{1,u} \\ B_t^{2,u} \\ B_t^{3,u}
		\end{pmatrix}G\varDelta x_t^u+\delta\begin{pmatrix}
			\varDelta X_t^u \\ \varDelta Y_t^u \\ \varDelta Z_t^u
		\end{pmatrix}^\top\begin{pmatrix}
			A_t^{11,u} & A_t^{12,u} & A_t^{13,u} \\
			A_t^{21,u} & A_t^{22,u} & A_t^{23,u} \\
			A_t^{31,u} & A_t^{32,u} & A_t^{33,u}
		\end{pmatrix}\begin{pmatrix}
			\varDelta x_t^u \\ \varDelta y_t^u \\ \varDelta z_t^u
		\end{pmatrix}\\
		+\delta K_1\big(\varDelta X_t^u\varDelta x_t^u+\varDelta Y_t^u\varDelta y_t^u+\varDelta Z_t^u\varDelta z_t^u\big)\Big\}\mathrm{d}t+(\cdots)\mathrm{d}W_t^u.
	\end{multline*}
	Integrating and using the monotonicity condition (\ref{monotonicity condition-1}), the boundedness of coefficients and the graphon, we obtain
	\begin{multline*}
		\mathbb{E}^\boxtimes[\varDelta X_T^\top\varDelta Y_T]\le \mathbb{E}^\boxtimes\biggl[ \int_0^T -K_1\big\{|\varDelta X_t|^2+|\varDelta Y_t|^2+|\varDelta Z_t|^2\big\}\mathrm{d}t \biggr]\\
		+\delta K_3 \mathbb{E}^\boxtimes\biggl[ \int_0^T \big\{|\varDelta X_t|^2+|\varDelta Y_t|^2+|\varDelta Z_t|^2+|\varDelta x_t|^2+|\varDelta y_t|^2+|\varDelta z_t|^2\big\}\mathrm{d}t \biggr].
	\end{multline*}
	By estimates for graphon-aggregated SDEs (Lemma \ref{existence and uniqueness of SDE graphon aggregation}) and general BSDEs,
	\begin{equation*}
        \begin{aligned}
		&\mathbb{E}^\boxtimes\biggl[ \sup_{0\le t\le T}|\varDelta X_t|^2 \biggr]\\
		&\le K\delta^2\mathbb{E}^\boxtimes\biggl[ \int_0^T\big\{|\varDelta x_t|^2+|\varDelta y_t|^2+|\varDelta z_t|^2\big\}\mathrm{d}t \biggr]+K\mathbb{E}^\boxtimes\biggl[ \int_0^T\big\{|\varDelta Y_t|^2+|\varDelta Z_t|^2\big\}\mathrm{d}t \biggr],\\
		&\mathbb{E}^\boxtimes\biggl[ \sup_{0\le t\le T}|\varDelta Y_t|^2+\int_0^T |\varDelta Z_t|^2\mathrm{d}t \biggr]\\
		&\le K\delta^2\mathbb{E}^\boxtimes\biggl[|\varDelta x_T|^2+ \int_0^T \big\{|\varDelta x_t|^2+|\varDelta y_t|^2+|\varDelta z_t|^2\big\}\mathrm{d}t \biggr]+K\mathbb{E}^\boxtimes\biggl[ |\varDelta X_T|^2+ \int_0^T |\varDelta X_t|^2\mathrm{d}t\biggr].
	    \end{aligned}
    \end{equation*}
	Combining these inequalities yields
	\begin{multline*}
		(1-\alpha)\mathbb{E}^\boxtimes[|\varDelta X_T|^2]+K_1\mathbb{E}^\boxtimes\biggl[ \int_0^T\big\{|\varDelta X_t|^2+|\varDelta Y_t|^2+|\varDelta Z_t|^2\big\}\mathrm{d}t \biggr]\\
		\le \delta K_3 \mathbb{E}^\boxtimes\biggl[ \int_0^T \big\{|\varDelta X_t|^2+|\varDelta Y_t|^2+|\varDelta Z_t|^2+|\varDelta x_t|^2+|\varDelta y_t|^2+|\varDelta z_t|^2\big\}\mathrm{d}t+|x_T|^2+|X_T|^2 \biggr].
	\end{multline*}
	Thus,
	\begin{multline*}
		K_1\mathbb{E}^\boxtimes\biggl[ \int_0^T\big\{|\varDelta X_t|^2+|\varDelta Y_t|^2+|\varDelta Z_t|^2\big\}\mathrm{d}t+|\varDelta X_T|^2 \biggr]\\
		\le \delta K_3 \mathbb{E}^\boxtimes\biggl[ \int_0^T \big\{|\varDelta X_t|^2+|\varDelta Y_t|^2+|\varDelta Z_t|^2+|\varDelta x_t|^2+|\varDelta y_t|^2+|\varDelta z_t|^2\big\}\mathrm{d}t+|\varDelta x_T|^2\biggr].
	\end{multline*}
	For sufficiently small $\delta$, we have
	\begin{multline*}
		\mathbb{E}^\boxtimes\biggl[ \int_0^T\big\{|\varDelta X_t|^2+|\varDelta Y_t|^2+|\varDelta Z_t|^2\big\}\mathrm{d}t+|\varDelta X_T|^2 \biggr]\\
		\le \frac{1}{2}\mathbb{E}^\boxtimes\biggl[ \int_0^T \big\{|\varDelta x_t|^2+|\varDelta y_t|^2+|\varDelta z_t|^2\big\}\mathrm{d}t+|\varDelta x_T|^2 \biggr],
	\end{multline*}
	proving that $\varPhi$ is a contraction for sufficiently small $\delta$.
\end{proof} 

\begin{proof}[Proof of Theorem \ref{existence and uniqueness of GMFLFBSDE}]
	Uniqueness: It follows directly from Lemma \ref{uniqueness of GMFLFBSDE}.
	
	Existence and Estimates: Consider the family of parameterized equations (\ref{GMFLFBSDE parameterized alpha}). Lemma \ref{existence and uniqueness of GMFLFBSDE parameterized alpha=0} establishes that equation (\ref{GMFLFBSDE parameterized alpha}) with $\alpha=0$ admits a unique solution for any given non-homogeneous terms $(b,\sigma,g)$. According to Lemma \ref{estimate of GMFLFBSDE alpha delta}, there exists a constant $\delta_0 > 0$ such that if, for some $\alpha_0$, the parameterized equation admits a unique solution for any $(b,\sigma,g)$, together with the solution estimate (\ref{estimate of GMFLFBSDE parameterized alpha0}), then for any $\alpha \in [\alpha_0, \alpha_0+\delta_0]$, the parameterized equation also admits a unique solution for any $(b,\sigma,g)$, along with the corresponding solution estimate (\ref{estimate of GMFLFBSDE parameterized alpha}). Since $\delta_0 > 0$ and is independent of $\alpha_0$, one can apply Lemma \ref{estimate of GMFLFBSDE alpha delta} iteratively a finite number of times. Consequently, it follows that equation (\ref{existence and uniqueness of GMFLFBSDE parameterized alpha=0}) with $\alpha=1$ admits a unique solution for any $(b,\sigma,g)$, and the solution satisfies the estimate (\ref{estimate of GMFLFBSDE}).
\end{proof}

	\subsection{Stability estimates for solutions to graphon-aggregated FBSDEs} 

This section investigates the stability of solutions to the linear graphon-aggregated FBSDE (\ref{GMFLFBSDE}) under perturbations of the graphon. Specifically, we establish an estimate quantifying the continuous dependence of the solution on the graphon.

\begin{thm}\label{stability of GFBSDE graphon}
	Assume that \(G_1, G_2 \in \mathcal{W}_0\), and that the linear graphon-aggregated FBSDE (\ref{GMFLFBSDE}) satisfies assumptions (S1) and (S2) under both graphons \(G_1\) and \(G_2\). Let \((X^1, Y^1, Z^1)\) and \((X^2, Y^2, Z^2)\) be the corresponding solutions. Then, the following stability estimate holds:
	\begin{multline}\label{estimate of continuous dependence GFBSDE}
		\mathbb{E}^\boxtimes\biggl[ \sup_{0\le t\le T}|X_t^1-X_t^2|^2+\sup_{0\le t\le T}|Y_t^1-Y_t^2|^2+\int_0^T |Z_t^1-Z_t^2|^2\mathrm{d}t \biggr]\\
		\le K\|G_1-G_2\|_\infty^2 \cdot \mathbb{E}^\boxtimes\biggl[ |x_0|^2+|h|^2+ \int_0^T\big\{|b_t|^2+|\sigma_t|^2+|g_t|^2\big\}\mathrm{d}t \biggr].
	\end{multline}
\end{thm}

\begin{proof}
	Suppose that \((X^1, Y^1, Z^1)\) and \((X^2, Y^2, Z^2)\) are the solutions to (\ref{GMFLFBSDE}) corresponding to the graphons \(G_1\) and \(G_2\), respectively. Then, the difference \((\Delta X, \Delta Y, \Delta Z) = (X^1 - X^2, Y^1 - Y^2, Z^1 - Z^2)\) satisfies the following equation:
	\begin{equation*}
		\begin{cases}
			\mathrm{d}X_t^u=\big(A_t^{21,u}X_t^u+A_t^{22,u}Y_t^u+A_t^{23,u}Z_t^u+B_t^{2,u}G_1X_t^u+B_t^{2,u}(G_1-G_2)X_t^{2,u}\big)\mathrm{d}t,\\
			\qquad\qquad +\big(A_t^{31,u}X_t^u+A_t^{32,u}Y_t^u+A_t^{33,u}Z_t^u+B_t^{3,u}G_1X_t^u+B_t^{3,u}(G_1-G_2)X_t^{2,u}\big)\mathrm{d}W_t^u,\\
			\mathrm{d}Y_t^u=\big(A_t^{11,u}X_t^u+A_t^{12,u}Y_t^u+A_t^{13,u}Z_t^u+B_t^{1,u}G_1X_t^u+B_t^{1,u}(G_1-G_2)X_t^{2,u}\big)\mathrm{d}t\\
            \qquad\qquad +Z_t^u\mathrm{d}W_t^u,\\
			X_0^u=x_0^u,\quad Y_T^u=G^{1,u}X_T^u+G^{2,u}G_1X_T^u,
		\end{cases}
	\end{equation*}
	By applying Theorem \ref{existence and uniqueness of GMFLFBSDE} and Proposition \ref{extension of domain of graphon operator}, which provide a priori estimates for solutions to such graphon-aggregated FBSDEs, we obtain
	\begin{equation*}
        \begin{aligned}
		&\mathbb{E}^\boxtimes\biggl[ \sup_{0\le t\le T}|X_t^1-X_t^2|^2+\sup_{0\le t\le T}|Y_t^1-Y_t^2|^2+\int_0^T |Z_t^1-Z_t^2|^2\mathrm{d}t \biggr]\\
		&\le K\mathbb{E}^\boxtimes\biggl[ \int_0^T |(G_1-G_2)X_t^{2}|^2\mathrm{d}t \biggr]\le K\|G_1-G_2\|_\infty^2 \cdot \mathbb{E}^\boxtimes\biggl[ \int_0^T |X_t^2|^2\mathrm{d}t \biggr]\\
		&\le K\|G_1-G_2\|_\infty^2 \cdot \mathbb{E}^\boxtimes\biggl[ |x_0|^2+|h|^2+ \int_0^T\big\{|b_t|^2+|\sigma_t|^2+|g_t|^2\big\}\mathrm{d}t \biggr].
	    \end{aligned}
    \end{equation*}
	The last inequality follows from the standard energy estimate for the solution \(X^2\) of the linear graphon-aggregated FBSDE under graphon \(G_2\), which is guaranteed by assumptions (S1) and (S2).
\end{proof}

\section{Concluding remarks}

In this paper, we have investigated an LQ Stackelberg stochastic graphon game involving one leader and a continuum of followers. The state equations of the followers interact through graphon coupling terms, with their diffusion coefficients depending on the individual state, the graphon aggregation term, and the control variables. The diffusion term of the leader's state equation depends on its own state and control variables. A rigorous mathematical framework is established, demonstrating that under the admissible control set, the states of the followers are essentially pairwise independent. We have established the Stackelberg-Nash equilibrium for the problem. In the followers' problem, given the leader's fixed strategy, we apply the maximum principle to characterize the Nash equilibrium. This leads to a graphon-aggregated FBSDE with a term that the followers' equilibrium states and adjoint states must satisfy. A sufficient condition for the existence and uniqueness of solutions to this class of FBSDEs is provided. In the leader's problem, under Assumption (A4) on the followers' interaction graphon, we aggregate the followers' states. By employing a duality method, we then derive a sufficient condition for the leader's Stackelberg equilibrium. For the followers' problem, we find that their Nash equilibrium is characterized by a class of graphon-aggregated FBSDEs. Using the method of continuity, we prove that under a certain monotonicity condition, the linear graphon-aggregated FBSDE admits a unique solution. Furthermore, we establish the continuity of the system's solution with respect to the interaction graphon.

This work has several limitations. First, the coefficients in the state equations and cost functionals of the followers (except for the graphon aggregation term) are assumed to be independent of the follower’s index $u$. Moreover, the control weight matrix for the followers is strictly positive definite. Additionally, the leader’s influence is confined to the cost functionals, as the leader’s state and control variables are not incorporated into the followers’ state equations. Future research will focus on generalizing the problem structure. For instance, by allowing the leader’s state or control variables to enter the followers’ dynamics and by introducing follower-dependent coefficients in the state equations. Further directions include establishing monotonicity conditions for generalized linear graphon-aggregated FBSDEs, as well as investigating the relationship between Stackelberg games with a finite number of followers and those involving a continuum of followers.

\end{document}